\documentclass[a4paper]{article}

\usepackage[utf8]{inputenc}
\usepackage{amsmath, amssymb, mathabx, mathtools,amsthm}
\usepackage{hyperref}
\usepackage{graphicx}
\usepackage{xcolor}
\usepackage[ruled, lined, longend, linesnumbered]{algorithm2e}
\usepackage[hmarginratio=3:2]{geometry}
\usepackage[shortlabels]{enumitem}
\usepackage[numbers,sort&compress]{natbib}

\newtheorem{definition}{Definition}[section]
\newtheorem{theorem}[definition]{Theorem}
\newtheorem{corollary}[definition]{Corollary}
\newtheorem{lemma}[definition]{Lemma}
\newtheorem{remark}[definition]{Remark}
\newtheorem{example}[definition]{Example}

\usepackage{chngcntr}
\counterwithin{figure}{section}
\counterwithin{table}{section}

\newcommand{\defas}{:=}
\newcommand{\ind}{\chi}

\newcommand{\nlOp}{\mathcal{L}}

\newcommand{\nlDom}{\Omega}
\newcommand{\nlBound}{\mathcal{I}}

\newcommand{\completeDom}{\nlDom \cup \nlBound}
\newcommand{\varOp}{A}
\newcommand{\varForce}{F}

\newcommand{\xb}{\mathbf{x}}
\newcommand{\yb}{\mathbf{y}}

\newcommand{\Vb}{\mathbf{V}}
\newcommand{\nb}{\mathbf{n}}

\newcommand{\Ub}{\mathbf{U}}
\newcommand{\Tb}{\mathbf{T}}

\newcommand{\kernel}{\gamma}
\newcommand{\kernelt}{\gamma^t}

\newcommand{\R}{\mathbb{R}}
\newcommand{\Rd}{\mathbb{R}^d}
\newcommand{\Hc}{\mathcal{H}}

\newcommand{\advar}{v}
\newcommand{\trialSpace}{V}
\newcommand{\testSpace}{V_c}

\newcommand{\weakSol}{u}

\newcommand{\transInvKernel}{J}
\newcommand{\grad}{\nabla}
\newcommand{\energyFunc}{E}
\newcommand{\AAMLagFunc}{G}
\newcommand{\objFunc}{J}
\newcommand{\data}{\bar{u}}

\newcommand{\g}{\gamma}

\newcommand{\shape}{\Gamma}
\newcommand{\shapet}{\shape^t}
\newcommand{\shapeFun}{J}
\newcommand{\shapeSpace}{\mathcal{A}}

\newcommand{\xt}{\xi^t}
\newcommand{\ballxy}{\chi_{B_\varepsilon(\xb)}(\yb)}
\def \N{\mathbb{N}}
\newcommand{\transformationMatrix}{\mathcal{B}}
\DeclareMathOperator{\di}{div}
\DeclareMathOperator{\tr}{trace}
\DeclareMathOperator{\supp}{supp}
\DeclareMathOperator{\dist}{dist}
\newcommand{\Ftb}{\mathbf{F}_{\mathbf{t}}}
\newcommand{\Ftzero}{\mathbf{F}_{0}}

\newcommand{\LtNSpace}{H}

\newcommand{\redFunc}{\objFunc^{red}}
\newcommand{\Ftbtilde}{\widetilde{\mathbf{F}}_t}

\newcommand{\vecfieldsspecific}{C_0^1(\completeDom,\R^d)}

\newcommand{\LtNOp}{\varOp^{LtN}}
\newcommand{\interfaceOp}{\LtNOp_{\shape}}
\newcommand{\interfaceOpt}{\LtNOp_{\shapet}}
\newcommand{\interfaceForce}{\varForce_{\shape}}
\newcommand{\shapeInnerProd}{b}
\newcommand{\lagrangian}{L}

\newcommand{\Id}{{\bf Id}}

\title{\textbf{Interface Identification constrained by Local-to-Nonlocal Coupling}}
\author{Matthias Schuster\thanks{Universitaet Trier, D-54286 Trier, Germany; Email: schusterm@uni-trier.de, volker.schulz@uni-trier.de}\hspace{2mm}\href{https://orcid.org/0000-0002-9355-1076}{\includegraphics[scale=0.06]{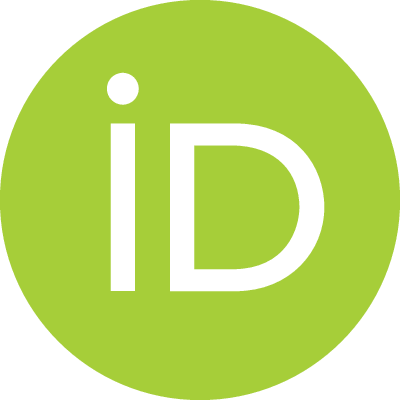}} \and Volker Schulz\footnotemark[1]\hspace{2mm}\href{https://orcid.org/0000-0001-7665-130X}{\includegraphics[scale=0.06]{orcid.eps}} }
\date{}

\begin{document}
	\maketitle
	\small
	\textbf{Abstract.}
	Models of physical phenomena that use nonlocal operators are better suited for some applications than their classical counterparts that employ partial differential operators. However, the numerical solution of these nonlocal problems can be quite expensive. Therefore, Local-to-Nonlocal couplings have emerged that combine partial differential operators with nonlocal operators. In this work, we make use of an energy-based Local-to-Nonlocal coupling that serves as a constraint for an interface identification problem.
	~
	\normalsize
	\section{Introduction}
	Nonlocal models have been successfully applied to a broad range of applications, e.g. anomalous or fractional diffusion\cite{brockmann2008,suzuki2022,deliaAnomalousTransport}, peridynamics\cite{silling2000reformulation, javili2019peridynamics}, finance and jump processes\cite{metzler2000random, delia2017nonlocal, capodaglio_energy_coupling, glusa2023asymptotically}, machine learning\cite{de2023machine, you2020data} and image denoising\cite{buades2010image, delia2021bilevel}.
	Typically, nonlocal operators are integral operators that are dependent on a function $\kernel$, which is called \emph{kernel}. As a result, nonlocal operators allow interactions between different points in space and therefore some physical phenomena are better described in this nonlocal regime. One drawback is, that the computation of solutions regarding nonlocal equations can be quite costly. If, e.g. we apply the finite element method we need to assemble a stiffness matrix that is typically dense for nonlocal operators. Additionally, the weak formulation, as we see in Section \ref{chap:Dir_Prob}, consists of double integrals, which are computationally more demanding than single integrals that typically represent weak formulations of PDEs. Since these stiffness matrices are sparse for PDEs and some physical processes can also be modeled by PDEs, the coupling of nonlocal and 'local' partial differential operators is one possibility to develop a model that is accurate but at the same time computationally not too expensive.
	For more information on these Local-to-Nonlocal couplings we refer to the review paper \cite{coupling_review} and the references therein.
	As a result, the question arises, on which part of the domain the nonlocal operator should be deployed, which leads us to shape optimization and specifically to an interface identification problem as a first step to model this problem. An overview on the topic of shape optimization can be found in, e.g. \cite{sokolowski_Introduction, shapes_geometries, allaire_shape, henrot_shape}, and a similar interface identification problem constrained by partial differential equations is investigated in \cite{welker_diss}. For an analysis of the interface identification problem constrained by completely nonlocal equations we refer to \cite{shape_paper}.\\
	The paper is organized as follows: First, we recall nonlocal Dirichlet problems and their well-posedness of the corresponding weak formulation. A nonlocal Dirichlet problem is part of the energy-based Local-to-Nonlocal(LtN) coupling that is outlined in Section \ref{section:LtN}, which is followed by a small subsection on how to solve this LtN coupling with Schwarz methods. After that, we introduce an interface identification problem that is in our case constrained by the energy-based LtN coupling of Section \ref{section:interface_problem_formulation}. Chapter \ref{section:shape_opt} is then dedicated to shape optimization and contains basic notions of shape optimization as well as the development of the shape derivative for the reduced objective functional corresponding to the interface identification problem from Section \ref{section:interface_problem_formulation}. Additionally, we describe a well-known shape optimization algorithm in Chapter \ref{section:shape_opt_algorithm}, which we apply in Section \ref{section:num_ex} to solve the interface identification problem constrained by an energy-based LtN coupling for two numerical examples.	
	
	\section{Nonlocal Terminology and Framework}
	\label{chap:Dir_Prob}
	Let $\nlDom \subset \Rd$ be an open, connected and bounded domain.
	Moreover, we denote by $\kernel:\Rd \times \Rd \rightarrow [0,\infty)$ a nonnegative \emph{(interaction) kernel} and we let the \emph{nonlocal boundary} $\nlBound$ of $\Omega$ be defined in the following way 
	\begin{align}
		\label{def:nlBound}
		\nlBound \defas \{\yb \in \Rd \setminus \nlDom: \int_{\nlDom} \kernel(\xb, \yb) + \kernel(\yb, \xb) ~d\xb > 0 \}.   
	\end{align}
	As in \cite{DuAnalysis, vollmann_diss}, the kernel $\kernel$ is supposed to fulfill the next two requirements.
	\begin{itemize}
		\item[(K1)] There exist $\delta \in (0,\infty)$ and $\phi:\Rd \times \Rd \rightarrow [0,\infty )$ with $$\kernel(\xb,\yb) = \phi(\xb,\yb) \ind_{B_\delta(\xb)}(\yb).$$ 
		\item[(K2)] There exist $\kernel_0 \in (0, \infty)$ and $\epsilon \in (0, \delta)$ such that
		\begin{align*}
			\kernel_0 \leq \kernel(\xb,\yb) \text{ for } \yb \in B_{\epsilon}(\xb) \text{ and } \xb \in \nlDom.
		\end{align*} 
	\end{itemize}
	Then, we call an operator $\nlOp$ of the form
	\begin{align}
		\label{op:conv-diff}
		-\nlOp u(\xb) = \int\limits_{\Rd} \left(u(\xb)\kernel(\xb,\yb) - u(\yb)\kernel(\yb,\xb) \right) ~d\yb
	\end{align}
	a \emph{nonlocal convection-diffusion operator}. 
	With an appropriate forcing term $f$ and suitable boundary data $g$, which will both be specified later, we are now able to formulate a \emph{nonhomogeneous steady-state Dirichlet problem with boundary constraints} as
	\begin{align}
		\begin{split}
			\label{nlProblem}
			-\nlOp u &= f \text{ on } \nlDom\\
			u &= g \text{ on } \nlBound.
		\end{split}
	\end{align}
	Similar to the case of partial differential equations, the theory of nonlocal problems focuses on finding a weak solution to \eqref{nlProblem}. In order to derive such a weak formulation we multiply the first equation of  \eqref{nlProblem} by a test function $\advar: \nlDom \cup \nlBound \rightarrow \R$ with $\advar = 0$ on $\nlBound$ and then integrate over $\nlDom$, which results in
	\begin{align}
		\left(-\nlOp u, \advar \right)_{L^2(\nlDom)} &= \left(f, \advar \right)_{L^2(\nlDom)} \nonumber \\
		\Leftrightarrow \int_{\nlDom} \advar(\xb)\int_{\Rd} (u(\xb)\kernel(\xb,\yb) &- u(\yb)\kernel(\yb,\xb)) ~d\yb d\xb = \int_{\nlDom} f(\xb) \advar(\xb) ~d\xb. \label{firstVarForm}
	\end{align}
	Making use of Fubini’s theorem and applying $\advar(\yb) = 0$ for $\yb \in \nlBound$ yields an equivalent representation for the integral on the left-hand side
	\begin{align*}
		&\int_{\nlDom} \advar(\xb)\int_{\completeDom} (u(\xb)\kernel(\xb,\yb) - u(\yb)\kernel(\yb,\xb)) ~d\yb d\xb\\ 
		&= \frac{1}{2} \iint_{(\completeDom)^2} (\advar(\xb) - \advar(\yb)) (u(\xb)\kernel(\xb,\yb) - u(\yb)\kernel(\yb,\xb)) ~d\yb d\xb.
	\end{align*}
	As a next step, we define the bilinear operator $\varOp$ and the linear functional $\varForce$ as follows
	\begin{align*}
		\varOp(u,v) &\defas \frac{1}{2} \iint_{(\completeDom)^2} (\advar(\xb) - \advar(\yb)) (u(\xb)\kernel(\xb,\yb) - u(\yb)\kernel(\yb,\xb)) ~d\yb d\xb \text{ and} \\
		\varForce(v) &\defas \int_{\nlDom} f(\xb) \advar(\xb) ~d\xb.
	\end{align*}
	Then, we can introduce a (semi-)norm
	\begin{align}
		\begin{split} \label{norm}
			|u|_{\trialSpace(\completeDom)} &\defas \sqrt{\varOp(u,u)} \text{ and a norm } \\ 
			||u||_{\trialSpace(\completeDom)} &\defas |u|_{\trialSpace(\completeDom)} + ||u||_{L^2(\completeDom)},
		\end{split}
	\end{align}
	which we employ to define the following nonlocal energy spaces
	\begin{align*}
		\trialSpace(\completeDom) &\defas \{u \in L^2(\completeDom): ||u||_{\trialSpace(\completeDom)} < \infty\} \text{ and} \\
		\testSpace(\completeDom) &\defas \{u \in \trialSpace(\completeDom): u = 0 \text{ on } \nlBound\}.
	\end{align*}
	With these definitions we can now formulate a \emph{variational} or \emph{weak} formulation of problem \eqref{nlProblem}:
	\begin{definition}
		Given $f \in L^2(\nlDom)$ and $g \in \trialSpace(\completeDom)$, if $\weakSol \in \trialSpace(\completeDom)$ solves
		\begin{align}\label{weak_formulation}
			\begin{split}
				\varOp(\weakSol, \advar) &= \varForce(\advar) \text{ for all } \advar \in \testSpace(\completeDom) \text{ and} \\
				u&-g \in \testSpace(\completeDom),
			\end{split}	
		\end{align}
		then $u$ is called weak solution of \eqref{nlProblem}.
	\end{definition}
	Since $\varOp(u,\advar) = \varOp(u - g,\advar) + \varOp(g,\advar)$ and $u-g \in \testSpace(\completeDom)$, we are able to rephrase problem \ref{weak_formulation} as a homogeneous Dirichlet problem
	\begin{align}
		\label{weak_formulation_homogeneous}
		\begin{split}
			\textit{Given } f &\in L^2(\nlDom) \textit{ and } g \in \trialSpace(\completeDom) \textit{, find } \tilde{\weakSol} \in \testSpace(\completeDom) \textit{ s.t.}\\
			\varOp(\tilde{\weakSol}, \advar) &= \varForce(\advar) - \varOp(g,\advar) \textit{ for all } \advar \in \testSpace(\completeDom),\\
			\textit{then } \weakSol &\defas \tilde{\weakSol} + g \in \trialSpace(\completeDom) \textit{ is a solution of \eqref{weak_formulation}.}
		\end{split}
	\end{align} 
	\begin{example}
		\label{remark:examples_well_posedness}
		Two popular classes of symmetric kernels, where the well-posedness of \eqref{weak_formulation} is shown in, e.g. \cite{DuAnalysis, vollmann_diss}, are:
		\begin{itemize}
			\item Integrable kernels:\\
			There exist constants $\kernel_1, \kernel_2 \in  (0, \infty)$ with
			\begin{align*}
				\kernel_1 \leq \inf_{\xb \in \nlDom} \int_{\completeDom} \kernel(\xb,\yb) ~d\yb \text{ and } \sup_{\xb \in \nlDom} \int_{\completeDom} \left( \kernel(\xb,\yb) \right)^2 ~d\yb \leq \left( \kernel_2 \right)^2.
			\end{align*}
			Then the solution $u \in \trialSpace(\completeDom)$ for \eqref{weak_formulation} exists, is unique and in this case the spaces $\left(\testSpace(\completeDom), |\cdot|_{\trialSpace(\completeDom)} \right)$ and $\left( L_c^2(\completeDom), ||\cdot||_{L^2(\completeDom)} \right)$ as well as $\left(\trialSpace(\completeDom), ||\cdot||_{\trialSpace(\completeDom)} \right)$ and $\left( L^2(\completeDom), ||\cdot||_{L^2(\completeDom)} \right)$ are equivalent.
			\item Singular symmetric kernels:\\
			There exist constants $\kernel_1, \kernel_2 \in (0, \infty)$ and $s \in (0,1)$ with
			\begin{align*}
				\kernel_1 < \kernel(\xb,\yb)||\xb - \yb||_2^{d+2s} < \kernel_2 \text{ for all } \yb \in B_{\delta}(\xb) \text{ and } \xb \in \nlDom.
			\end{align*}
			Then the solution to \eqref{weak_formulation} exists, is unique and the spaces $\left(\testSpace(\completeDom), |\cdot|_{\trialSpace(\completeDom)} \right)$ and \\ $\left( H^s_c(\completeDom), |\cdot|_{H^s(\completeDom)} \right)$ as well as $\left(\trialSpace(\completeDom), ||\cdot||_{\trialSpace(\completeDom)} \right)$ and $\left(H^s(\completeDom), ||\cdot||_{H^s(\completeDom)} \right)$ are equivalent.
		\end{itemize}
		In the coupling formulation, that is presented in the Chapter \ref{section:LtN}, only integrable kernels with some additional assumptions, which will be specified later, are considered.
	\end{example}
	In the next Chapter, we present a Local-to-Nonlocal coupling, where we will observe a nonlocal Dirichlet problem on one subdomain. Additionally, we recapture under which assumptions this particular coupling is well-posed(see Lemma \ref{lemma:AcostaNormEquivalence} and Corollary \ref{cor:LtN_well_posedness}).

	\section{Interface Identification constrained by an energy-based Local-to-Nonlocal Coupling}
	\subsection{Energy-based Local-to-Nonlocal Coupling}
	\label{section:LtN}
	In this section we introduce a specific coupling of the Laplacian operator on one domain and a nonlocal operator on the other domain, which was first formulated in \cite{Acosta_coupling, AcostaDD}.\\ 
	Here, for any two sets $\hat{\nlDom}_1, \hat{\nlDom}_2 \subset \Rd$ we define the distance of those two domains as 
	\begin{align*}
		\dist(\hat{\nlDom}_1, \hat{\nlDom}_2) \defas \inf\limits_{\xb \in \hat{\nlDom}_1,~ \yb \in \hat{\nlDom}_2} ||\xb - \yb||_2.
	\end{align*}
	If $\hat{\nlDom}_1$ only consists of one element, i.e. there exists one $\xb \in \Rd$ such that $\hat{\nlDom}_1 = \{\xb\}$, we also use the short notation $\dist(\xb,\hat{\nlDom}_2) \defas \dist(\{\xb\},\hat{\nlDom}_2)$.\\
	From now on, the set $\nlDom$ is assumed to be an open and bounded Lipschitz domain and can be described as $\nlDom = \left(\overline{\nlDom}_l \cup \overline{\nlDom}_{nl}\right)^\circ$, where $\nlDom_{l}$ and $\nlDom_{nl}$ are open, bounded, connected, non-empty and disjoint sets. Moreover, $\nlDom_l$ is also supposed to be a Lipschitz domain and, since trivial couplings should be excluded, $\dist(\nlDom_l, \nlDom_{nl}) < \delta$ is assumed to hold. An example configuration is shown in Picture \ref{fig:LtN_setting}. From now on, we will refer to $\nlDom_{l}$ as the 'local' and to $\nlDom_{nl}$ as the 'nonlocal' domain.
	\begin{figure}[h!]
		\centering
		\def\svgwidth{0.35\textwidth}
		{\small 
\begingroup%
  \makeatletter%
  \providecommand\color[2][]{%
    \errmessage{(Inkscape) Color is used for the text in Inkscape, but the package 'color.sty' is not loaded}%
    \renewcommand\color[2][]{}%
  }%
  \providecommand\transparent[1]{%
    \errmessage{(Inkscape) Transparency is used (non-zero) for the text in Inkscape, but the package 'transparent.sty' is not loaded}%
    \renewcommand\transparent[1]{}%
  }%
  \providecommand\rotatebox[2]{#2}%
  \newcommand*\fsize{\dimexpr\f@size pt\relax}%
  \newcommand*\lineheight[1]{\fontsize{\fsize}{#1\fsize}\selectfont}%
  \ifx\svgwidth\undefined%
    \setlength{\unitlength}{425.97277063bp}%
    \ifx\svgscale\undefined%
      \relax%
    \else%
      \setlength{\unitlength}{\unitlength * \real{\svgscale}}%
    \fi%
  \else%
    \setlength{\unitlength}{\svgwidth}%
  \fi%
  \global\let\svgwidth\undefined%
  \global\let\svgscale\undefined%
  \makeatother%
  \begin{picture}(1,0.99973671)%
    \lineheight{1}%
    \setlength\tabcolsep{0pt}%
    \put(0,0){\includegraphics[width=\unitlength,page=1]{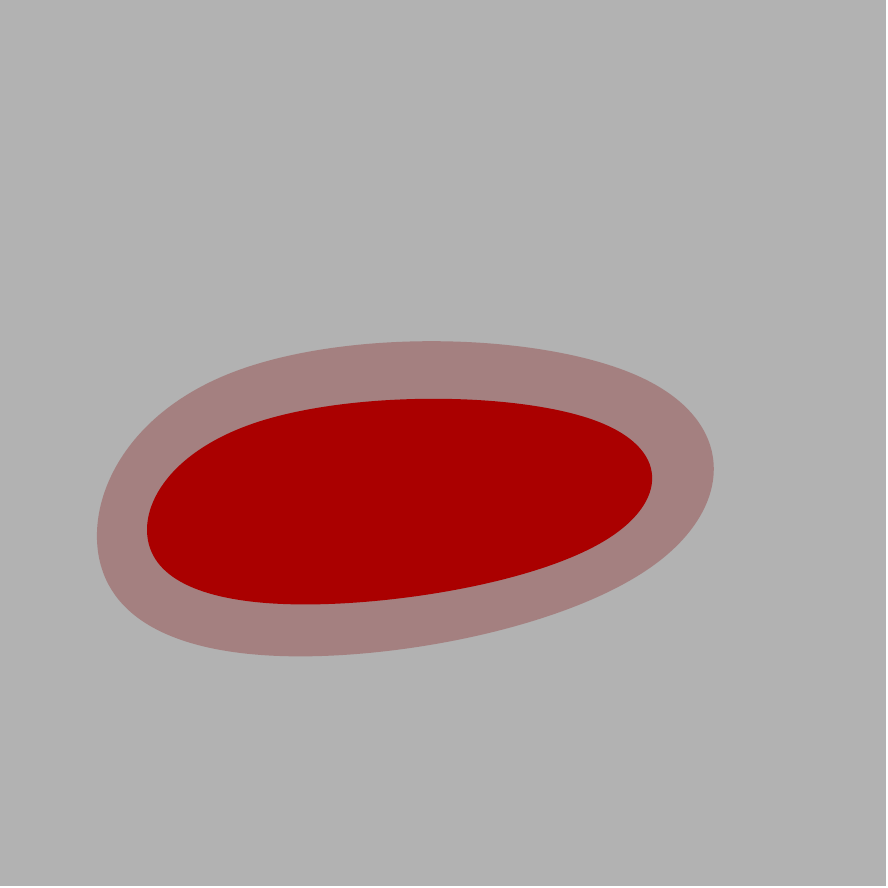}}%
    \put(0.38467928,0.43889045){\color[rgb]{0,0,0}\makebox(0,0)[lt]{\lineheight{1.25}\smash{\begin{tabular}[t]{l}$\nlDom_{nl}$\end{tabular}}}}%
    \put(0.30,0.80){\color[rgb]{0,0,0}\makebox(0,0)[lt]{\lineheight{1.25}\smash{\begin{tabular}[t]{l}$\nlDom_l$\end{tabular}}}}%
    \put(0.45,0.56){\color[rgb]{0,0,0}\makebox(0,0)[lt]{\lineheight{1.25}\smash{\begin{tabular}[t]{l}$\nlBound_{nl}$\end{tabular}}}}%
  \end{picture}%
\endgroup%
}
		\caption{Here, the nonlocal domain $\nlDom_{nl}$ is depicted in the darker red and the corresponding nonlocal domain $\nlBound_{nl}$ of the nonlocal domain $\nlDom_{nl}$ is colored in the lighter red. Moreover, the 'local' domain $\nlDom_l$ consists in this example of the nonlocal boundary $\nlBound_{nl}$ as well as the gray area.}
		\label{fig:LtN_setting}
	\end{figure}
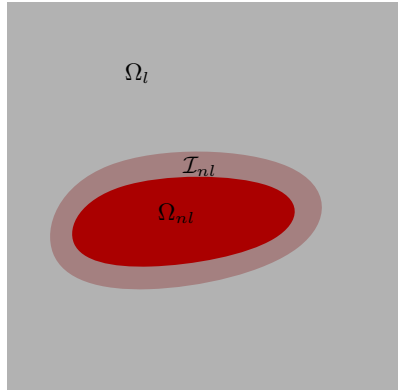
	In the remaining part of this work we denote by $\partial \nlDom$, $\partial \nlDom_l$ and $\partial \nlDom_{nl}$ the \emph{local boundary} of $\nlDom$, $\nlDom_l$ and $\nlDom_{nl}$, respectively. Moreover, the nonlocal boundary $\nlBound_{nl}$ regarding the nonlocal domain $\nlDom_{nl}$ is analogously defined as in \eqref{def:nlBound}, i.e.
	\begin{align*}
		\nlBound_{nl} \defas \{\yb \in \Rd\setminus\nlDom_{nl}: \int_{\nlDom_{nl}} \kernel(\xb,\yb) + \kernel(\yb,\xb) ~d\yb d\xb > 0\}.
	\end{align*}
	Then, the coupling problem is formulated in the following way:
	\begin{align}
		- \nlOp_l (\weakSol)(\xb) &= f(\xb) && \text{for } \xb \in \nlDom_l, \nonumber \\
		\partial_n \weakSol_l &= 0 && \text{on } \partial \nlDom_l \cap \partial \nlDom_{nl}, \label{local_LtN_problem} \\
		\weakSol_l &= 0 && \text{on } \partial \nlDom_l \cap \partial \nlDom, \nonumber \\ \nonumber \\
		-\nlOp_{nl}(\weakSol)(\xb) &= f(\xb) && \text{for } \xb \in \nlDom_{nl}, \nonumber \\
		\weakSol_{nl} &= \weakSol_l && \text{on } \nlBound_{nl} \cap \nlDom_l, \label{nonlocal_LtN_problem} \\
		\weakSol_{nl} &= 0 && \text{on } \nlBound_{nl} \cap \nlBound, \nonumber
	\end{align}
	where 
	\begin{align*}
		- \nlOp_l (\weakSol)(\xb) \defas -\nlOp_{l}(\weakSol_l, \weakSol_{nl})(\xb) \defas& - \Delta \weakSol_l(\xb) + \int_{\nlDom_{nl}} \left( \weakSol_l(\xb) - \weakSol_{nl}(\yb) \right) \kernel(\xb, \yb) ~d\yb \quad \text{and}\\
		-\nlOp_{nl}(\weakSol)(\xb) \defas -\nlOp_{nl}(\weakSol_{l}, \weakSol_{nl})(\xb) \defas& 2\int\limits_{\nlDom_{nl}} \left( \weakSol_{nl}(\xb) - \weakSol_{nl}(\yb) \right) \kernel(\xb,\yb) ~d\yb \\
		&+ \int\limits_{\nlBound_{nl} \cap \nlDom_l} \left( \weakSol_{nl}(\xb) - \weakSol_{l}(\yb) \right) \kernel(\xb,\yb) ~d\yb + \int\limits_{\nlBound_{nl} \cap \nlBound} \weakSol_{nl}(\xb) \kernel(\xb,\yb) ~d\yb.
	\end{align*}
	Basically, we solve a nonlocal Dirichlet problem on the nonlocal domain $\nlDom_{nl}$, where the 'local' solution $\weakSol_l$ serves as boundary data on $\nlBound_{nl} \cap \nlDom_{l}$. For the 'local' problem \eqref{local_LtN_problem} the dependence on the nonlocal solution $\weakSol_{nl}$ is more implicit. 
	Here, the operator $\nlOp_l (\weakSol)(\xb)$ is equivalent to the Laplacian $\Delta \weakSol_l(\xb)$, if $\dist(\xb, \nlDom_{nl}) \geq \delta$ and otherwise, if $\dist(\xb,\nlDom_{nl}) < \delta$, some nonlocal flux is added to the Laplacian. Thus, in the second case the value of the operator $\nlOp_l$ is influenced by the 'nonlocal' solution $\weakSol_{nl}$. Since $\kernel = 0$ almost everywhere in $\nlDom_{nl} \times \left(\Rd \setminus  \left(\nlDom_{nl} \cup \nlBound_{nl}\right) \right)$ we sometimes integrate over $\nlDom_{l}$ and $\nlBound$ instead of $\nlBound_{nl} \cap \nlDom_{l}$ and $\nlBound_{nl} \cap \nlBound$ in the remaining part of this work.
	\begin{remark}
		According to \cite{AcostaDD} the assumption that $\nlDom_{nl}$ needs to be connected can be relaxed to $\hat{\delta}$-connectedness for some $\hat{\delta} \in [0,\delta)$. A domain $\hat{\nlDom}$ is said to be $\hat{\delta}$-connected, if $\hat{\nlDom}$ cannot be written as a union of two disjoint open sets, whose distance is greater than $\hat{\delta}$, i.e.
		\begin{align*}
			\nexists \hat{\nlDom}_1,\hat{\nlDom}_2 \text{ open}:\ \hat{\nlDom}_1 \cap \hat{\nlDom}_2 = \emptyset,\ \hat{\nlDom}_1 \cup \hat{\nlDom}_2 = \hat{\nlDom} \text{ and } \dist(\hat{\nlDom}_1, \hat{\nlDom}_2) > \hat{\delta}.
		\end{align*}
		However, since $\hat{\delta}$-connectedness has to be checked after every deformation of $\shape$, we only consider the case, where $\nlDom_{nl}$ is (0-)connected. Thus, we do not apply the general concept of $\hat{\delta}$-connectedness. 
	\end{remark}
	In order to formulate a weak formulation of the LtN coupling we define the LtN bilinear operator $\LtNOp$ and the linear functional $\varForce$ as 
	\begin{align*}
		\varForce (\advar) &\defas \int_{\nlDom_l}f(\xb) v_l(\xb) ~d\xb - \int_{\nlDom_{nl}} f(\xb) v_{nl}(\xb) ~d\xb \quad \text{and}\\
		\LtNOp (\weakSol, \advar) &\defas \int_{\nlDom_l} \grad u_l(\xb) \grad v_l(\xb) ~d\xb
		+ \int_{\nlDom_{nl}} \int_{\nlDom_{nl}} \left( u_{nl}(\xb) - u_{nl}(\yb) \right) \left( v_{nl}(\xb) - v_{nl}(\yb) \right) \kernel(\xb,\yb) ~d\yb d\xb 
		\\
		&+ \int_{\nlDom_{nl}} \int_{\nlDom_l} (u_{nl}(\xb) - u_l(\yb)) (v_{nl}(\xb) - v_l(\yb))\kernel(\xb,\yb) ~d\yb d\xb \\
		&+ \int_{\nlDom_{nl}}u_{nl}(\xb)v_{nl}(\xb) \int_{\nlBound} \kernel(\xb,\yb) ~d\yb d\xb.  
	\end{align*}
	Then we can define an inner product $<\cdot,\cdot\cdot>_{\LtNSpace}$ and the corresponding norm $||\cdot||_{\LtNSpace}$ as 
	\begin{align*}
		<\weakSol,\advar>_{\LtNSpace} \defas \LtNOp(\weakSol, \advar) \quad \text{and} \quad ||\weakSol||_{\LtNSpace} \defas \sqrt{\LtNOp(\weakSol, \weakSol)}.
	\end{align*}
	\begin{definition}
		The energy space $(\LtNSpace, ||\cdot||_{\LtNSpace})$  is defined as
		\begin{align*}
			\LtNSpace &\defas \{\weakSol=(\weakSol_l, \weakSol_{nl}): \weakSol_l \in H_{\partial \nlDom \cap \partial \nlDom_l }^1(\nlDom_l),\ \weakSol_{nl} \in L_c^2\left( \nlDom_{nl} \cup \nlBound_{nl} \right) \},
		\end{align*}
		where the spaces $H_{\partial \nlDom \cap \partial \nlDom_l }^1(\nlDom_l)$ and $L_c^2\left( \nlDom_{nl} \cup \nlBound_{nl} \right)$ are set as follows
		\begin{align*}
			H_{\partial \nlDom \cap \partial \nlDom_l }^1(\nlDom_l) &\defas \{\advar \in H^1(\nlDom_l): \advar = 0 \text{ on }  \partial \nlDom \cap \partial \nlDom_{l} \} \quad \text{and}\\
			L_c^2\left( \nlDom_{nl} \cup \nlBound_{nl} \right) &\defas \{\advar \in L^2(\nlDom_{nl} \cup \nlBound_{nl}): \advar = 0 \text{ on }  \nlBound_{nl}\}.
		\end{align*}
	\end{definition}
	
	\begin{definition}
		Given $f \in L^2(\nlDom)$, we call $\weakSol = (\weakSol_l, \weakSol_{nl}) \in \LtNSpace$ a weak solution to the coupling that consists of equations \eqref{local_LtN_problem} and \eqref{nonlocal_LtN_problem}, if $\weakSol$ fulfills
		\begin{align}
			\label{LtN_weak_form}
			\varOp^{LtN}(\weakSol,\advar) = \varForce(\advar) \text{ for all } \advar=(\advar_l, \advar_{nl}) \in \LtNSpace.
		\end{align}
	\end{definition}
	Thus, equation \eqref{LtN_weak_form} is the variational formulation of the combined problem.
	Next, we would like to point out, why this coupling is called energy-based.
	Therefore, we want to minimize the energy functional
	\begin{align*}
		\min_{\weakSol \in \LtNSpace}\energyFunc(\weakSol) &\defas \frac{1}{2} \LtNOp(\weakSol,\weakSol) - \varForce(\advar),
	\end{align*}
	and derive the first order necessary condition
	\begin{align*}
		d\energyFunc(\weakSol)(\advar) = \varOp^{LtN}(\weakSol,\advar) - \varForce(\advar) = 0 \text{ for all } \advar=(\advar_l, \advar_{nl}) \in \LtNSpace.
	\end{align*}
	which is equivalent to the variational formulation \eqref{LtN_weak_form}.
	Now, if we set the local component of the test vector $\advar_l = 0$, we get the variational formulation regarding the nonlocal Dirichlet problem \eqref{nonlocal_LtN_problem} of the LtN coupling
	\begin{align*}
		&\int_{\nlDom_{nl}} \int_{\nlDom_{nl}} \left( \advar_{nl}(\xb) - \advar_{nl}(\yb) \right) \left( \weakSol_{nl}(\xb) - \weakSol_{nl}(\yb) \right) \kernel(\xb,\yb) ~d\yb d\xb \\
		&+ \int_{\nlDom_{nl}} \advar_{nl}(\xb) \int_{\nlDom_l} \left( \weakSol_{nl}(\xb) - \weakSol_l(\yb) \right) \kernel(\xb,\yb) ~d\yb d\xb
		+ \int_{\nlDom_{nl}} \advar_{nl}(\xb)\weakSol_{nl}(\xb) \int_{\nlBound} \kernel(\xb,\yb) ~d\yb d\xb \\
		&= \int_{\nlDom_{nl}} f(\xb)\advar_{nl}(\xb) ~d\xb \text{ for all } \advar_{nl} \in \testSpace(\nlDom_{nl} \cup \nlBound_{nl}) = L^2_c(\nlDom_{nl} \cup \nlBound_{nl}).
	\end{align*}
	Analogously, by setting the nonlocal component of the test vector $\advar_{nl} = 0$, we derive the weak formulation for the local subproblem \eqref{local_LtN_problem} of the LtN coupling 
	\begin{align*}
		\int_{\nlDom_l} \grad \weakSol_l(\xb) \grad \advar_l(\xb) ~d\xb + \int_{\nlDom_l} \advar_l(\xb) \int_{\nlDom_{nl}} \left( \weakSol_l(\xb) - \weakSol_{nl}(\yb) \right)\kernel(\xb, \yb) ~d\yb d\xb \\
		= \int_{\nlDom_l}f(\xb)\advar_l(\xb) ~d\xb
		\text{ for all } \advar_l \in H_{\partial \nlDom_l \cap \partial \nlDom}^1(\nlDom_l).
	\end{align*}
	All in all, the coupling is called energy-based, because the weak formulations for \eqref{local_LtN_problem}, \eqref{nonlocal_LtN_problem} and the combined problem can be derived by investigating first order necessary conditions for the minimization of a specific energy functional, where we only consider certain test functions.\\
	For the purpose of having a well-posed LtN coupling, we need further restrictions. In addition to the assumptions (K1) and (K2) from Chapter \ref{chap:Dir_Prob}, the kernel $\kernel$ is supposed to be
	\begin{itemize}
		\item[(K3)] bounded and
		\item[(K4)] translation invariant, i.e. there exists a function $\transInvKernel: \Rd \rightarrow \R$ and $$\kernel(\xb, \yb) = J(\xb - \yb).$$
	\end{itemize}
	\begin{remark}
		\label{remark:comparison_kernel_Acosta}
		In contrast to this work, in \cite{AcostaDD} the kernel $\kernel$ is not assumed to be truncated. Instead they assume additionally the following conditions:
		\begin{itemize}
			\item The kernel $\kernel$ is integrable.
			\item (Compactness): The operator $C: L^2(\nlDom_{nl}) \rightarrow L^2(\nlDom_{nl})$ with $$C(u)(\xb) = \int_{\nlDom_{nl}}\transInvKernel(\xb - \yb)u(\yb) ~d\yb$$ is compact.
		\end{itemize}
		Then, in \cite[Lemma 2.6]{AcostaDD} it is shown that the space $(\LtNSpace, <\cdot,\cdot\cdot>_{\LtNSpace})$ is actually a Hilbert space.\\
		In our case the truncatedness and boundedness of $\kernel$ imply $\kernel \in L^1((\nlDom_{nl} \cup \nlBound_{nl}) \times (\nlDom_{nl} \cup \nlBound_{nl}))$ and $\kernel \in L^2((\nlDom_{nl} \cup \nlBound_{nl}) \times (\nlDom_{nl} \cup \nlBound_{nl}))$ since the domain $\nlDom_{nl}$ and therefore $\nlDom_{nl} \cup \nlBound_{nl}$ are also bounded. As a result, the compactness condition holds, because the operator $C$ is in this case a Hilbert-Schmidt operator and therefore compact(see \cite[Theorem 6.12]{Brezis}).
	\end{remark}
	Consequently, we derive the next Lemma.
	\begin{lemma}[{After \cite[Lemma 2.6]{AcostaDD}}]
		\label{lemma:AcostaNormEquivalence}
		Let $\kernel$ fulfill (K1)-(K4). Then, the space $(\LtNSpace,<\cdot,\cdot\cdot>_{\LtNSpace})$ is a Hilbert space and the norm $||\cdot||_{\LtNSpace}$ is equivalent to the norm of the product space $\left(H^1(\nlDom_l), ||\cdot||_{H^1(\nlDom_l)}\right)\times \left(L_c^2(\nlDom_{nl} \cup \nlBound_{nl}), ||\cdot||_{L^2(\nlDom_{nl})}\right)$.
	\end{lemma}
	\begin{proof}
		Follows directly from \cite[Lemma 2.6]{AcostaDD} and Remark \ref{remark:comparison_kernel_Acosta}.
	\end{proof}
	\begin{corollary}
		\label{cor:LtN_well_posedness}
		As a consequence of Lemma \ref{lemma:AcostaNormEquivalence} and the Riesz representation theorem, there exists a unique solution to \eqref{LtN_weak_form}.
	\end{corollary}
	Moreover, we will make use of Lemma \ref{lemma:AcostaNormEquivalence} in order to show that the averaged adjoint method, which will be introduced in Section \ref{section:AAM}, can be applied. Next, we present a way to iteratively compute the solution to \eqref{LtN_weak_form}.
	\subsection{Schwarz Methods}
	\label{section:schwarz}
	In Chapter \ref{section:num_ex}, we solve the energy-based LtN coupling by applying the multiplicative Schwarz method as described in \cite{AcostaDD}. Schwarz methods were developed by H.A. Schwarz around 1870 in order to solve a PDE on an 'irregular' domain\cite{schwarz1869}. In the following, we present the Schwarz method in the context of the LtN coupling problem.
	The basic idea is to solve problems \eqref{local_LtN_problem} and \eqref{nonlocal_LtN_problem} by finding solutions to the local and the nonlocal problem in an alternating fashion as it is shown in Algorithm \ref{alg:mult_schwarz}.\\
	\begin{algorithm}\label{alg:mult_schwarz}
		\textbf{Given}: $\weakSol^0=(\weakSol_l^0, \weakSol_{nl}^0)$.\\
		\While{termination criterion not fulfilled}{
			Find $\weakSol_l^k$, s.t.
			\begin{alignat*}{2}
				-\nlOp_l(\weakSol_l^k, \weakSol_{nl}^{k-1}) &= f \quad && \text{on } \nlDom_l, \\
				\partial_n \weakSol_l^k &= 0 && \text{on } \partial \nlDom_l \cap \partial \nlDom_{nl},\\
				\weakSol_l^k &= 0 && \text{on } \partial \nlDom_l \cap \nlDom.
			\end{alignat*}
			Find $\weakSol_{nl}^k$, s.t.
			\begin{alignat*}{2}
				-\nlOp_{nl}(\weakSol_{l}^k ,\weakSol_{nl}^k) &= f && \text{on } \nlDom_{nl}, \\
				\weakSol_{nl}^k &= \weakSol_l^k \quad && \text{on } \nlBound_{nl} \cap \nlDom_l,\\
				\weakSol_{nl}^k &= 0 && \text{on } \nlBound_{nl} \cap \nlBound.
			\end{alignat*}
			$k \leftarrow k+1$.
		}
		\caption{Multiplicative Schwarz Method for LtN coupling}
	\end{algorithm}
	Given the solution $\weakSol^{k-1}=(\weakSol^{k-1}_l, \weakSol^{k-1}_{nl})$ from the previous iteration, we first compute the solution $\weakSol_l^k$ of the local problem with $\weakSol^{k-1}_{nl}$ as a second argument of $-\nlOp_l$ and after that, we find $\weakSol_{nl}^k$ by solving the nonlocal problem with boundary data $\weakSol_l^k$. So, we use the new solution $\weakSol_l^k$ immediately in the same iteration.
	\begin{theorem}[{\cite[Lemma 2.9]{AcostaDD}}]
		Let $\weakSol = (\weakSol_l, \weakSol_{nl}) \in \LtNSpace$ be a solution to \eqref{LtN_weak_form}.
		Algorithm \ref{alg:mult_schwarz} yields a sequence $\weakSol^k=(\weakSol^k_l, \weakSol^k_{nl})$, such that
		\begin{align*}
			\weakSol_l^k \rightarrow \weakSol_l \text{ in } H^1_{\partial \nlDom \cap \nlDom_l}(\nlDom_l) \text{ and}\\
			\weakSol_{nl}^k \rightarrow \weakSol_{nl} \text{ in } L^2_c(\nlDom_{nl} \cup \nlBound_{nl}).  
		\end{align*} 
		Additionally, this method is geometrically convergent, i.e. there exists a constant $\varepsilon \in (0,1)$ with
		\begin{align*}
			\max\{||\weakSol_l^k - \weakSol_l||_{H^1_{\partial \nlDom \cap \partial \nlDom_l}(\nlDom_l)}, ||\weakSol_{nl}^k - \weakSol_{nl}||_{L^2(\nlDom_{nl})} \} \leq C \varepsilon^k,
		\end{align*}
		where $C \defas \max\{||\weakSol_l^0 - \weakSol_l||_{H^1_{\partial \nlDom \cap \partial \nlDom_l}(\nlDom_l)}, ||\weakSol_{nl}^0 - \weakSol_{nl}||_{L^2(\nlDom_{nl})} \}$.
	\end{theorem}
	Alternatively, we could also follow the approach of the \emph{additive Schwarz method} and employ only the former iterate $\weakSol^{k-1}$ inside one iteration and switch to $\weakSol^{k}$ as the current solution after the iteration, which enables us to solve both subproblems in parallel.\\
	In both cases, we repeat this alternating process until some termination criterion is met. In this work, since we apply the finite element method to solve the subproblems, we stop when the Euclidean norm of the residual is smaller than some given tolerance.
	
	\subsection{Problem Formulation}
	\label{section:interface_problem_formulation}
	Interface identification is a popular problem in PDE-constrained shape optimization, see, e.g. \cite{welker_diss, Sturm_diss, shapes_geometries}. Here, as mentioned earlier, we consider an open, bounded and connected domain $\nlDom$ with Lipschitz domain as it is depicted in Picture \ref{fig:LtN_setting}. We then call $\shape \defas \partial \nlDom_{nl} \cap \partial \nlDom_{l}$ \emph{interface} and assume that $\shape$ is a closed curve that decomposes $\nlDom$ into the Lipschitz domains $\nlDom_{nl}$ and $\nlDom_{l}$, i.e. $\nlDom = \nlDom_{nl} \dot{\cup} \shape \dot{\cup} \nlDom_l$. Therefore, we denote by $\nlDom(\shape)$ that $\nlDom$ is decomposed by $\shape$ as described above. Moreover, the space $\LtNSpace$ and the norm $||\cdot||_{\LtNSpace}$ are also dependent on $\shape$, which we indicate by $||\cdot||_{\LtNSpace(\shape)}$ and $\LtNSpace(\shape)$, respectively.\\
	From now on, let $f_{\shape} \defas f_{l} \ind_{\nlDom_l} + f_{nl} \ind_{\nlDom_{nl}}$ with $f_l, f_{nl} \in H^1(\nlDom)$. Given some data $\data \in H^1(\nlDom)$, that we would like to approximate by a weak solution $\weakSol$ of the LtN coupling \eqref{LtN_weak_form}, we then call
	\begin{align}
		\begin{split}\label{interface_problem}
			\min_{\weakSol, \shape} \objFunc(\weakSol, \shape) &\defas j(\weakSol, \Gamma) + j_{per}(\shape) = \int_{\nlDom(\shape)} \left(\weakSol(\xb) - \data(\xb)\right)^2 ~d\xb + \nu \int_{\shape} 1 ~ds \\
			\text{s.t.} \quad &\interfaceOp(\weakSol,\advar) = \interfaceForce(\advar) \text{ for all } \advar \in \LtNSpace(\shape),
		\end{split}
	\end{align}
	an \emph{interface identification problem}, which is in our case constrained by a Local-to-Nonlocal coupling. Here, the subscripts of $\varForce_{\shape}$ and $\LtNOp_{\shape}$ indicate the dependence of these functions on $\shape$.
	Since every admissible choice of $\shape \subset \nlDom$ yields a unique solution, this problem can be seen as only dependent on the choice of the interface $\shape$, i.e. we would like to find a shape $\shape$, such that the corresponding solution approximates our data $\data$ as good as possible. The second term in the objective functional $ j_{per}$ is called perimeter regularization with parameter $\nu > 0$ and is often applied in optimization to avoid an ill-posedness of the problem\cite{perimeter}.
	
	\section{Shape Optimization}
	\label{section:shape_opt}
	In this Chapter, we follow the steps in \cite{shape_paper}, where an interface identification constrained by purely nonlocal equations is investigated.
	\subsection{Basics of Shape Optimization}
	In order to introduce the definition of a shape derivative we need a family of mappings $\{\Ftb\}_{t \in [0,T]}$ with $\Ftb:\overline{\nlDom} \rightarrow \Rd$ and $\Ftzero = \Id$ for a given sufficiently small constant $T \in (0, \infty)$. We can then deform a domain $\shape \subset \nlDom$ according to $\{\Ftb\}_{t \in [0,T]}$ and get a family of perturbed shapes $\{\shapet\}_{t\in [0,T]}$, where $\shapet \defas \Ftb(\shape)$.
	In this work, as a family of mappings we will only use the so-called \emph{perturbation of identity}.
	\begin{definition}[Perturbation of Identity]
		For a vector field $\Vb \in C_0^k(\nlDom, \R)$, $k \in \N$ we define the \emph{perturbation of identity} as
		\begin{align}
			\Ftb: \overline{\nlDom} \rightarrow \R, \quad \Ftb(\xb) = \xb + t\Vb(\xb) = \left(\Id + t\Vb \right)(\xb),
		\end{align}
		where $\Id:\Rd \rightarrow \Rd$ is the identity function.
	\end{definition}
	In Figure \ref{fig:pert_id_explained} we present an example how the perturbation of identity maps $\shape$ to a new interface $\Ftb(\shape)$.
	\begin{figure}[h!]
		\centering
		\def\svgwidth{0.75\textwidth}
		{\small 
\begingroup%
  \makeatletter%
  \providecommand\color[2][]{%
    \errmessage{(Inkscape) Color is used for the text in Inkscape, but the package 'color.sty' is not loaded}%
    \renewcommand\color[2][]{}%
  }%
  \providecommand\transparent[1]{%
    \errmessage{(Inkscape) Transparency is used (non-zero) for the text in Inkscape, but the package 'transparent.sty' is not loaded}%
    \renewcommand\transparent[1]{}%
  }%
  \providecommand\rotatebox[2]{#2}%
  \newcommand*\fsize{\dimexpr\f@size pt\relax}%
  \newcommand*\lineheight[1]{\fontsize{\fsize}{#1\fsize}\selectfont}%
  \ifx\svgwidth\undefined%
    \setlength{\unitlength}{344.37659985bp}%
    \ifx\svgscale\undefined%
      \relax%
    \else%
      \setlength{\unitlength}{\unitlength * \real{\svgscale}}%
    \fi%
  \else%
    \setlength{\unitlength}{\svgwidth}%
  \fi%
  \global\let\svgwidth\undefined%
  \global\let\svgscale\undefined%
  \makeatother%
  \begin{picture}(1,0.41491243)%
    \lineheight{1}%
    \setlength\tabcolsep{0pt}%
    \put(0,0){\includegraphics[width=\unitlength,page=1]{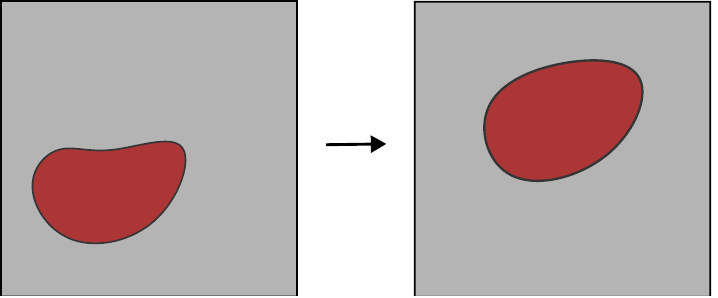}}%
    \put(0.48,0.22808561){\color[rgb]{0,0,0}\makebox(0,0)[lt]{\lineheight{1.25}\smash{\begin{tabular}[t]{l}$\Ftb$\end{tabular}}}}%
    \put(0,0){\includegraphics[width=\unitlength,page=2]{pert_id_explained.pdf}}%
    \put(0.13,0.21){\color[rgb]{0,0,0}\makebox(0,0)[lt]{\lineheight{1.25}\smash{\begin{tabular}[t]{l}$\xb$\end{tabular}}}}%
    \put(0.72,0.347){\color[rgb]{0,0,0}\makebox(0,0)[lt]{\lineheight{1.25}\smash{\begin{tabular}[t]{l}$\Ftb(\xb)$\end{tabular}}}}%
    \put(0.82,0.13916668){\color[rgb]{0,0,0}\makebox(0,0)[lt]{\lineheight{1.25}\smash{\begin{tabular}[t]{l}$\Ftb(\shape)$\end{tabular}}}}%
    \put(0.23,0.08331695){\color[rgb]{0,0,0}\makebox(0,0)[lt]{\lineheight{1.25}\smash{\begin{tabular}[t]{l}$\shape$\end{tabular}}}}%
  \end{picture}%
\endgroup%
}
		\caption{The perturbation of identity moves a point $\xb$ in the direction of a vector field $\Vb$ to the point $\Ftb(\xb)$. Applied on $\shape$, this interface is shifted and deformed to a new interface $\Ftb(\shape)$.}
		\label{fig:pert_id_explained}
	\end{figure}
	If $t$ is sufficiently small, such that $||t\Vb||_{W^{1,\infty}(\nlDom,\Rd)} < 1$, then the perturbation of identity $\Ftb:\overline{\nlDom} \rightarrow \overline{\nlDom}$ is Lipschitz, invertible and the inverse is also Lipschitz\cite{henrot_shape}. As a result, if $\shape$ is a Lipschitz domain, the perturbed domain $\shapet$ will also have a Lipschitz boundary. We are now able to present the definition of a shape derivative.
	\begin{definition}[Shape Derivative]
		Given a shape functional $\shapeFun:\shapeSpace \rightarrow \R$, where $\shapeSpace$ is an appropriate shape space, the \emph{Eulerian} or \emph{directional derivative} is defined as
		\begin{align}
			D_{\shape} \shapeFun(\shape)[\Vb] = \lim_{t \searrow 0} \frac{\shapeFun(\Ftb(\shape)) - \shapeFun(\shape)}{t}.
		\end{align}
		If for every $\Vb \in C_0^k(\nlDom, \Rd)$ the derivative $D_{\shape} J(\shape)[\Vb]$ exists and the function $\Vb \mapsto D_{\shape} \shapeFun (\shape) [\Vb]$ is linear and continuous, the derivative $D_{\shape} \shapeFun (\shape) [\Vb]$ is called \emph{shape derivative} at $\shape$ in the direction $\Vb$.
	\end{definition}
	In this work, we use the same framework as in \cite{linear_view} and, starting from an initial shape $\shape^0$, we define the corresponding shape space $\shapeSpace(\shape^0)$ as follows: 
	\begin{definition}[Shape Space]
		Given a domain $\shape^0 \subset \nlDom$ we set the corresponding \emph{shape space} $\shapeSpace(\shape^0)$ as
		\begin{align*}
			\shapeSpace(\shape^0) \defas \{\Tb(\shape^0)|~ \Tb:\overline{\nlDom} \rightarrow \overline{\nlDom} \text{ invertible} \}.
		\end{align*}
	\end{definition} 
	Since the composition $\Tb^n \circ \Tb^{n-1} \circ ... \circ \Tb^1$ of invertible functions $\Tb^k:\overline{\nlDom} \rightarrow \overline{\nlDom}$ for $k=1,...,n$ and $n \in \N$ is again an invertible function, the shape space $\shapeSpace(\shape^0)$ contains only domains that can be obtained by a finite number of suitable deformations of $\shape^0$. Because we only employ the perturbation of identity, the function $\Tb^k = \Id + t\Vb$ is invertible, if $||t\Vb||_{W^{1,\infty}(\nlDom,\Rd)} < 1$ for $k \in \N$ as mentioned above.
	\begin{remark}
		Since it could be the case that parts of the nonlocal boundary $\nlBound_{nl}$ are outside of $\nlDom$, i.e. $\nlBound \cap \nlBound_{nl} \neq \emptyset$, we extend every $\Vb \in C_0^k(\nlDom,\Rd)$ to $\overline{\nlDom} \cup \nlBound$ by zero. So, the boundary $\nlBound$ stays unchanged and only points inside $\nlDom$ can be moved according to $$\Vb \in C_0^k(\completeDom, \Rd) \defas \{\Vb:\overline{\nlDom} \cup \nlBound \rightarrow \Rd|~ \Vb = 0 \text{ on } \nlBound \text{ and } \left. \Vb \right|_{\overline{\nlDom}} \in C_0^k(\nlDom,\Rd) \}.$$ In the following, we only need vector fields that are one times continuously differentiable, i.e. $\Vb \in \vecfieldsspecific$.
	\end{remark}
	\subsection{Shape Derivative via the Averaged Adjoint Method}
	\label{section:AAM}
	As mentioned in Section \ref{section:interface_problem_formulation}, we assume the existence of a unique solution $\weakSol(\shape)$ for every admissible shape $\shape$, i.e. $\weakSol(\shape)$ satisfies $\interfaceOp(u(\shape), \advar) = \interfaceForce( \advar)$ for all $\advar \in \LtNSpace(\shape)$. Therefore, we investigate the so-called \emph{reduced problem}
	\begin{align}
		\label{eq_reduced_functional}
		\min\limits_{\shape \in \shapeSpace(\shape^0)}\; &\redFunc(\shape) := \objFunc(\weakSol(\shape), \shape).
	\end{align}
	Since we would like to apply derivative based minimization algorithms, we need to compute the shape derivative of the reduced objective functional $J^{red}$. There are several ways how the shape derivative of the objective functional can be computed, which are discussed in \cite{Sturm_diss}. 
	We choose the \emph{averaged adjoint method (AAM)} of \cite{sturm_minimax,AAM,Sturm_diss}, which is a material derivative-free approach to rigorously access the shape derivative of \eqref{eq_reduced_functional}. 
	Therefore, we need the so-called \emph{Lagrangian functional}
	\begin{align*}
		\lagrangian(\weakSol,\shape,\advar) \defas \objFunc(\weakSol,\shape) + \interfaceOp(\weakSol,\advar) - \interfaceForce(\advar).
	\end{align*}
	With the Lagrangian we can express the reduced functional as 
	\begin{align*}
		\redFunc(\shape)= \lagrangian(\weakSol(\shape),\shape,\advar), \quad \text{for an arbitrary } \advar \in \LtNSpace(\shape),
	\end{align*}
	where we apply that $\weakSol(\shape)$ is a weak solution of the LtN coupling \eqref{LtN_weak_form} regarding the interface $\shape$. We now fix $\shape$ and denote by $\shapet \defas \Ftb(\shape)$, $\nlDom_l^t \defas \Ftb(\nlDom_l)$ and $\nlDom_{nl}^t \defas \Ftb(\nlDom_{nl})$ the deformed interior boundary and the deformed domains, respectively. Furthermore we indicate by writing $\nlDom(\shapet)$ that we use the decomposition $\nlDom(\shapet) = \nlDom_1^t \dot{\cup} \shapet \dot{\cup} \nlDom_2^t \left( = \nlDom \right)$. As a consequence, the norm $||\cdot||_{\LtNSpace(\shapet)}$ and therefore the space $\LtNSpace(\shapet)$ differ from $||\cdot||_{\LtNSpace(\shape)}$ and $\LtNSpace(\shape)$, since $\shapet(=\partial \nlDom_l^t \cap \partial \nlDom_{nl}^t)$ indicates how $\nlDom$ is decomposed into $\nlDom_l^t$ and $\nlDom_{nl}^t$, which determines the integration domains that are employed to describe the operator $\LtNOp_{\shape}$. Then, the reduced objective functional regarding $\shapet$ can be written as
	\begin{align}
		\label{deformed_problem}
		\redFunc(\shapet)= \lagrangian(\weakSol(\shapet),\shapet,\advar), \quad \text{for any } \advar \in \LtNSpace(\shapet),
	\end{align}
	where $\weakSol(\shapet) \in \LtNSpace(\shapet)$.
	In order to differentiate $L$ with respect to $t$ to compute the shape derivative, we would have to access the derivative for $u(\shapet) \circ \Ftb$ and $\advar \circ \Ftb$, where $u(\shapet),\advar \in \LtNSpace(\shapet)$ may not be differentiable. Moreover, the norm $||\cdot||_{\LtNSpace(\shapet)}$ and consequently the space $\LtNSpace(\shapet)$ are also dependent on $t$. Instead, we can circumvent differentiating $\weakSol(\shapet) \circ \Ftb$ and additionally stay in $\left(\LtNSpace(\shape), ||\cdot||_{\LtNSpace(\shape)}\right)$ by using $\Ftb$, which is a homeomorphism, as a so-called \emph{pull-back function}. Then, for $\weakSol,\advar \in \LtNSpace(\shapet)$ there exist functions $\tilde{u},\tilde{\advar} \in \LtNSpace(\shape)$, such that
	\begin{align*}
		u = \tilde{u} \circ \Ftb^{-1} ~\text{and}~ \advar = \tilde{\advar} \circ \Ftb^{-1}.
	\end{align*}
	Moreover, for a sufficiently small $T \in (0, \infty)$, we can define
	\begin{align}
		J:[0,T]\times &\LtNSpace(\shape)\rightarrow \R,\nonumber\\
		&J(t,u) \defas J(u \circ \Ftb^{-1},\shapet), \nonumber\\
		\LtNOp:[0,T]\times &\LtNSpace(\shape)\times \LtNSpace(\shape) \rightarrow \R, \nonumber\\
		&\LtNOp(t,u,\advar)\defas \interfaceOpt(u\circ \Ftb^{-1}, \advar \circ \Ftb^{-1}), \nonumber\\
		\varForce:[0,T]\times &\LtNSpace(\shape)\rightarrow \R,\nonumber\\
		&F(t,\advar)\defas F_{\shapet}(\advar \circ \Ftb^{-1}), \nonumber\\
		\begin{split}
			G:[0,T] \times &\LtNSpace(\shape) \times \LtNSpace(\shape) \rightarrow \R,\\
			&G(t,u,\advar) \defas L(u \circ \Ftb^{-1},\shapet,\advar \circ \Ftb^{-1})
			= J(t, u)  + \LtNOp(t,u,\advar) - \varForce(t, \advar). \label{AAM_function}
		\end{split}
	\end{align}
	Now, let us briefly mention that the Local-to-Nonlocal bilinear form regarding $\shapet$ can be written as
	\begin{align}
		\begin{split} \label{eq:transformed_LtN_Op}
			&\LtNOp(t, \weakSol^0,\advar^0) = \interfaceOpt(\weakSol^0 \circ \Ftb^{-1}, \advar^0 \circ \Ftb^{-1}) = \int\limits_{\nlDom_l} \left(\transformationMatrix(t)(\xb) \grad \weakSol^0(\xb), \grad \advar^0(\xb) \right) ~d\xb \\
			&+ \int\limits_{\nlDom_{nl}} \int\limits_{\nlDom_{nl} \cup \nlBound_{nl}} \left(\advar^0(\xb) - \advar^0(\yb)\right) \left(\weakSol^0(\xb) - \weakSol^0(\yb)\right)\kernelt(\xb,\yb)\xt(\xb) \xt(\yb) ~d\yb d\xb,
		\end{split}
	\end{align}
	where 
	\begin{align*}
		\xt &\defas \det D\Ftb,\ \kernelt(\xb,\yb) \defas \kernel(\Ftb(\xb),\Ftb(\yb)) \text{ and }\\
		\transformationMatrix(t)(\xb) &\defas \xt(\xb) D\Ftb^{-1}(\Ftb(\xb))D\Ftb^{-\top}(\Ftb(\xb)) = \xt(\xb) \left(D\Ftb(\xb)\right)^{-1}\left(D\Ftb(\xb)\right)^{-\top}.
	\end{align*}
	Then, \eqref{deformed_problem} can be expressed as
	\begin{align*}
		J^{red}(\shapet)= G(t,u^t,\advar),\quad \text{for any } \advar \in \LtNSpace(\shape),
	\end{align*}
	where $u^t \in \LtNSpace(\shape)$ is defined as the unique solution of the LtN coupling \eqref{LtN_weak_form} for $\shapet$, i.e. $\weakSol^t$ solves
	\begin{align}
		\label{LtN_weak_form_shape_dependent}
		\LtNOp(t,\weakSol^t,\advar) - \varForce(t,\advar) = 0,\quad \forall \advar \in \LtNSpace(\shape).
	\end{align}
	Obviously, $\LtNOp(t,u,\advar) - \varForce(t,\advar)$ is linear in $\advar$ for all $(t,u) \in [0,T] \times \LtNSpace(\shape)$, which is one necessary condition of AAM.
	Additionally, the following prerequisites have to hold in order to apply AAM. 
	\begin{itemize}
		\item \textbf{Assumption (H0)}:
		For every $(t,\advar) \in [0,T] \times \LtNSpace(\shape)$
		\begin{enumerate}
			\item $[0,1] \ni s \mapsto G(t,su^t+(1-s)u^0,\advar)$ is absolutely continuous and
			\item $[0,1] \ni s \mapsto d_uG(t,su^t + (1-s)u^0,\advar)[\tilde{u}] \in L^1((0,1))$ for all $\tilde{u} \in \LtNSpace(\shape)$.
		\end{enumerate}
		\item For every $t \in [0,T]$ there exists a unique solution $\advar^t \in \LtNSpace(\shape)$ for the averaged adjoint equation
		\begin{equation}
			\label{AAE}
			\int_0^1 d_u G(t,su^t + (1-s)u^0,\advar^t)[\tilde{u}]ds = 0 \quad \text{for all } \tilde{u} \in \LtNSpace(\shape).
		\end{equation}
		\item \textbf{Assumption (H1)}:\\
		The following equation holds
		\begin{equation*}
			\lim_{t \searrow 0} \frac{G(t,u^0,\advar^t)- G(0,u^0,\advar^t)}{t}= \partial_t G(0,u^0,\advar^0).
		\end{equation*}
	\end{itemize}
	Since $\LtNOp(t, \tilde{u}, \advar^t)$ is linear in the second argument, the left-hand side of the averaged adjoint equation \eqref{AAE} can be written as
	\begin{align*}
		\int_0^1 d_u G(t,su^t + (1-s)u^0,\advar^t)[\tilde{u}] ~ds = \LtNOp(t,\tilde{u},\advar^t) + \int_\Omega \left(\frac{1}{2}(u^t + u^0) - \bar{u}^t\right)\tilde{u}\xt ~d\xb,
	\end{align*}
	where $\bar{u}^t(\xb) \defas \bar{u}(\Ftb(\xb))$. Therefore, \eqref{AAE} is equivalent to
	\begin{align}
		\label{AAE_equiv_form}
		\LtNOp(t, \tilde{u}, \advar^t) = - \int_{\Omega} \left(\frac{1}{2}\left( u^t + u^0 \right) - \bar{u}^t \right) \tilde{u}\xt ~d\xb \quad \forall \tilde{u} \in \LtNSpace(\shape).
	\end{align}
	We will make use of \eqref{AAE_equiv_form} in the proof of Lemma \ref{lemma:AAM_reduced_functional}, where we show that, under some mild additional requirements, Assumptions (H0) and (H1) are fulfilled.
	Moreover, for $t=0$ we call the averaged adjoint equation
	\begin{align}
		\label{eq:adjoint}
		\LtNOp(0,\tilde{u},\advar^0) &= - \int_\Omega (u^0 - \bar{u})\tilde{u} ~d\xb \quad \forall \tilde{u} \in \LtNSpace(\shape)
	\end{align}
	adjoint equation and the function $\advar^0$, which solves \eqref{eq:adjoint}, adjoint solution. So we omit the word 'averaged' in this case. 
	Moreover, the LtN problem \eqref{LtN_weak_form_shape_dependent} for $t=0$ is titled state equation and the solution $u^0$ is named state solution.
	Finally, we can state how to derive the shape derivative of the reduced functional.
	\begin{theorem}[{\cite[Theorem 3.1]{AAM}}]
		\label{theo:AAM}
		Let the Assumptions (H0) and (H1) be fulfilled, $\weakSol^t$ solve \eqref{LtN_weak_form_shape_dependent} and let $\advar^t$ be the unique solution to the averaged adjoint equation (\ref{AAE}) for $t \in [0,T]$. Then, for $\advar \in \LtNSpace(\shape)$ we obtain
		\begin{align}
			\label{shape_derivative_reduced_functional}
			D_\shape J^{red}(\Gamma)[\Vb]=\left.\frac{d}{dt}\right|_{t=0^+}J^{red}(\shapet)=\left.\frac{d}{dt}\right|_{t=0^+}G(t,u^t,\advar) = \partial_t G(0,u^0,\advar^0).
		\end{align}
	\end{theorem}
	\begin{proof}
		See proof of \cite[Theorem 3.1]{AAM}.
	\end{proof}
	\subsection{Deriving the Shape Derivative of the Reduced Functional}
	In order to rigorously compute the shape derivative we make use of the following Lemma.
	\begin{lemma}[{After \cite[Proposition 2.32]{sokolowski_Introduction}}]
		\label{lemma_frechet_diff_bounded_domain}
		Let $D \subset \Rd$ be a bounded domain with nonzero measure, $f \in W^{1,1}(D,\R)$ and $\Vb \in C_{0}^1(D, \Rd)$, where $k \in \N$. Moreover, assume that $T > 0$ is sufficiently small, such that $\Ftb: \overline{D} \rightarrow \overline{D}$ is bijective for all $t \in [0,T]$. Then, $t \mapsto f \circ \Ftb$ is differentiable in $L^1(D)$ with
		\begin{align*}
			\left.\frac{d}{dt}\right|_{t=0} \left(f \circ \Ftb\right) = \grad f^{\top} \Vb.
		\end{align*} 
	\end{lemma}
	\begin{proof}
		Since $C^{\infty}(D)$ is dense in $W^{1,1}(D,\R)$\cite[Theorem 3.17]{adams}, we only need to show the result for $f \in C^{\infty}(D)$.
		Now, given $f \in C^{\infty}(D)$, applying the mean value theorem yields
		\begin{align*}
			f(\Ftb(\xb)) - f(\xb) = \int_{0}^{1} t\grad f(\xb + st\Vb(\xb))^{\top} \Vb(\xb) ~ds
		\end{align*}
		for $\xb \in D$. Consequently, we derive
		\begin{align}
			&\int_{D}\left|\frac{1}{t}\left(f(\Ftb(\xb)) - f(\xb)\right) - \grad f(\xb)^{\top} \Vb(\xb)\right| ~d\xb
			\leq \int_{D}  \int_{0}^{1} \left| \left( \grad f(\xb + st\Vb(\xb)) - \grad f(\xb)\right)^{\top} \Vb(\xb) \right|~ds ~d\xb \nonumber \\
			&\leq \int_{0}^{1} \int_{D} \left| \left( \grad f(\xb + st\Vb(\xb)) - \grad f(\xb)\right)^{\top} \Vb(\xb) \right| ~d\xb ~ds, \label{proof_lemma_frechet}
		\end{align}
		where we changed the order of integration in the last step. Now, we just have to show, that the double integral \eqref{proof_lemma_frechet} vanishes, when $t \rightarrow 0$. 
		For that reason we would like to show, that the inner integral converges to zero. 
		Since $\Vb \in C_0^1(D,\Rd)$ is continuous and thus bounded on $\overline{D}$ and $\grad f(\Ftb(\xb))$ converges in $L^2(D,\Rd)$  and therefore also in $L^1(D,\Rd)$ to $\grad f$ due to Lemma \ref{lemma:l2_convergence}, we conclude by employing the dominated convergence theorem that
		\begin{align*}
			\lim_{t \rightarrow 0} \int_{D} \left| \left( \grad f(\xb + st\Vb(\xb)) - \grad f(\xb)\right)^{\top} \Vb(\xb) \right| ~d\xb = 0 \quad \text{for all } s \in [0,1].
		\end{align*}
		Again, applying the dominated convergence theorem yields
		\begin{align*}
			\lim_{t \rightarrow 0} \int_{0}^{1} \int_{D} \left| \left( \grad f(\xb + st\Vb(\xb)) - \grad f(\xb)\right)^{\top} \Vb(\xb) \right| ~d\xb ~ds = 0.
		\end{align*}
	\end{proof}
	\begin{corollary}
		\label{cor:frechet_diff}
		Let $D \subset \Rd \times \Rd$ be an open and bounded domain with nonzero measure. If $\g \in W^{1,1}(D,\R)$ and $\tilde{\Vb} \in C_0^1(D, \Rd \times \Rd)$, then $t \mapsto \g \circ \Ftbtilde$, where $\Ftbtilde\defas (\xb,\yb) + t\tilde{\Vb}(\xb,\yb)$, is Fr{\'e}chet differentiable in $L^1(D,\R)$ at $t=0$ and its derivative is given by
		\begin{align*}
			\left.\frac{d}{dt}\right|_{t=0} \g \circ \Ftbtilde = \nabla \g^T \tilde{\Vb}.
		\end{align*}
	\end{corollary}
	\begin{remark}
		The statement of Lemma \ref{lemma_frechet_diff_bounded_domain} and Corollary \ref{cor:frechet_diff} is still valid, if $\Vb \in C_0^1(\hat{D},\Rd)$, where the superset $\hat{D} \supset D$ is a bounded and open domain.
	\end{remark}
	In our case, we set $\tilde{\Vb}(\xb,\yb) \defas \left(\Vb(\xb),\Vb(\yb) \right)$ in order to use Corollary \ref{cor:frechet_diff} to derive the shape derivative of the LtN-operator $\LtNOp_{\shape}$. Now, we are able to prove that the assumptions of the averaged adjoint method hold for the interface identification problem constrained by the energy-based LtN coupling under an additional condition.
	\begin{theorem}
		\label{lemma:AAM_reduced_functional}
		Let the kernel $\kernel$ be weakly differentiable and assume this weak derivative to be essentially bounded, i.e. $\kernel \in W^{1,\infty}(\left(\completeDom\right) \times \left( \completeDom \right))$. Moreover, $\kernel$ is assumed to fulfill (K1)-(K4) and we define $\AAMLagFunc$ as in \eqref{AAM_function}. Then, the Assumptions (H0) and (H1) of the averaged adjoint method hold and there exists a unique solution to the averaged adjoint equation \eqref{AAE} for every $t \in [ 0, T ]$.
	\end{theorem}
	\begin{proof}
		See Appendix \ref{appendix:proof_AAM}.
	\end{proof}
	\begin{corollary}
		Let $\kernel \in W^{1,\infty}(\left(\completeDom\right) \times \left( \completeDom \right))$ fulfill (K1)-(K4) and let the functions $\AAMLagFunc$ and $\redFunc$ be defined as in \eqref{AAM_function} and \eqref{eq_reduced_functional}, respectively. Then, the shape derivative of the reduced cost functional exists and can be computed by
		\begin{align*}
			D_{\shape} \redFunc(\shape)[\Vb] = \partial_t \AAMLagFunc(0, \weakSol^0, \advar^0),	
		\end{align*}
		where $\weakSol^0$ is the solution to the state equation \eqref{LtN_weak_form_shape_dependent} and $\advar^0$ solves the adjoint equation \eqref{eq:adjoint}.
	\end{corollary}
	In the remaining part of this subsection we would like to derive the shape derivative of the reduced functional 
	\begin{align*}
		D_{\shape} \redFunc(\shape)[\Vb] = D_{\shape} \objFunc(\weakSol^0, \shape)[\Vb] + D_{\shape} \interfaceOp(\weakSol^0, \advar^0)[\Vb] - D_{\shape} \interfaceForce(\advar^0)[\Vb] 
	\end{align*}
	by separately computing the shape derivative of the objective functional $\objFunc$, the bilinear operator $\LtNOp_{\shape}$ and the linear functional $\varForce_{\shape}$.
	\begin{theorem}
		\label{theo:shape_der_1}
		Let $\objFunc$ and $\varForce_{\shape}$ be defined as before. Then the shape derivatives of $\objFunc$ and $\varForce_{\shape}$ in the direction of a vector field $\Vb \in \vecfieldsspecific$ can be expressed as
		\begin{align*}
			D_{\shape} \objFunc(\weakSol^0)[\Vb] &= \int_{\nlDom} - \left( \weakSol^0 - \data \right) \grad \data^{\top} \Vb + \left(\weakSol^0 - \data \right)^2 \di \Vb ~d\xb + \nu \int_{\shape} \di \Vb - \nb^{\top}\grad \Vb^\top \nb ~ds \text{ and} \\
			D_{\shape} \interfaceForce(\advar^0)[\Vb] &= \int_{\nlDom} \left(\grad f_{\shape}^{\top} \Vb \right) \advar^0 + f_{\shape} \advar^0 \di \Vb ~d\xb.
		\end{align*}
		\begin{proof}
			First, from, e.g. \cite{schmidt_diss}, we derive $\xi^0(x)= \det D\Ftzero(\xb)= \det(\Id)=1$ and $\left. \frac{d}{dt}\right|_{t = 0^+} \xt = \di \Vb$. Thus, the shape derivative of the right-hand side $F_\shape$ can be computed by applying Lemma \ref{lemma_frechet_diff_bounded_domain} and the product rule of Fr\'{e}chet derivatives as follows
			\begin{align*}
				D_\shape F_\shape(\advar^0)[\Vb]  &= \left. \frac{d}{dt}\right|_{t = 0^+} F_{\shapet}(\advar^0 \circ \Ftb^{-1}) = \int_\Omega \left. \frac{d}{dt}\right|_{t = 0^+} (f_\shape \circ \Ftb) \advar^0 \xt ~d\xb \\ 
				&= \int_{\nlDom} \left( \grad f_\shape^{\top} \Vb \right) \advar^0 ~d \xb +\int_{\nlDom}f_\shape \advar^0 ~\di \Vb~d\xb.
			\end{align*}
			Moreover, the shape derivative of the objective functional can be written as
			\begin{align*}
				D_{\shape}J(u^0, \shape)[\Vb] = D_{\shape}j(u^0, \shape)[\Vb] + D_{\shape}j_{per}(\shape)[\Vb] = \left. \frac{d}{dt} \right|_{t = 0^+} j(u^0 \circ \Ftb^{-1}, \shapet) + \left. \frac{d}{dt} \right|_{t = 0^+} j_{per}(\shapet).
			\end{align*}
			Here, the shape derivative of the regularization term is given by\cite[Theorem 4.13]{welker_diss} 
			\begin{align*}
				D_{\shape}j_{per}(u^0,\shape)[\Vb]  = \nu\int_{{\shape}}\di_{\shape} \Vb ~ds 
				= \nu\int_{{\shape}}\di  \Vb  - \nb^\top \nabla \Vb^\top  \nb~ds,
			\end{align*}
			where $\nb$ denotes the outer normal of $\partial \Omega_{nl} \cap \partial \nlDom_{l}$. Additionally, we obtain for the shape derivative of the tracking-type functional
			\begin{align*}
				D_\shape j(u^0,\shape)[\Vb] &= \left. \frac{d}{dt}\right|_{t = 0^+} j(u^0 \circ \Ftb^{-1}, \shapet) = \frac{1}{2} \left. \frac{d}{dt}\right|_{t = 0^+} \int_{\Ftb(\nlDom)} (u^0 \circ \Ftb^{-1} - \bar{u})^2 ~d\xb\\ 
				&= \frac{1}{2} \int_{\nlDom} \left. \frac{d}{dt}\right|_{t = 0^+} (u^0 - \bar{u} \circ \Ftb)^2 \xt ~d\xb = \int_{\nlDom} -(u^0 - \bar{u})\nabla\bar{u}^\top\Vb + (u^0 - \bar{u})^2 \di \Vb ~d\xb.
			\end{align*}
		\end{proof}
	\end{theorem}
	\begin{theorem}[Shape Derivative of the Local-to-Nonlocal Operator]
		\label{theo:shape_der_2}
		Let $\LtNOp_{\shape}$ be defined as before and let $\kernel \in W^{1,\infty}(\left(\completeDom\right) \times \left(\completeDom\right))$ fulfill (K1)-(K4). Then, for a vector field $\Vb \in \vecfieldsspecific$ we can compute the shape derivative via
		\begin{align}
			D_{\shape} &\interfaceOp(\weakSol^0,\advar^0)[\Vb] = \int\limits_{\nlDom_l} - \left( \left( \grad \Vb(\xb) + \grad \Vb^\top(\xb) \right) \grad \weakSol^0(\xb), \grad \advar^0(\xb) \right) + \left( \grad \weakSol^0(\xb), \grad \advar^0(\xb) \right) \di \Vb(\xb) ~d\xb \nonumber \\
			&+ \int\limits_{\nlDom_{nl}} \int\limits_{\nlDom_{nl} \cup \nlBound_{nl}} \left(\advar^0(\xb) - \advar^0(\yb)\right) \left(\weakSol^0(\xb) - \weakSol^0(\yb)\right)\left(\grad_{\xb} \kernel(\xb,\yb)^\top \Vb(\xb) + \grad_{\yb} \kernel(\xb,\yb)^\top \Vb(\yb) \right) ~d\yb d\xb \nonumber \\
			&+ \int\limits_{\nlDom_{nl}} \int\limits_{\nlDom_{nl} \cup \nlBound_{nl}} \left(\advar^0(\xb) - \advar^0(\yb)\right) \left(\weakSol^0(\xb) - \weakSol^0(\yb)\right)\left(\di \Vb(\xb) + \di \Vb(\yb) \right) ~d\yb d\xb. \label{eq:shape_der_LtN_op}
		\end{align}
	\end{theorem}
	\begin{proof} 
		We compute the shape derivative of the LtN-operator \eqref{eq:transformed_LtN_Op} in two steps. Therefore, we start by differentiating the nonlocal part of $\LtNOp$, i.e. the double integral, and then we develop the shape derivative of the remaining integral.\\ In, e.g. \cite{schmidt_diss} it is shown, that $\xt$ is continuously differentiable with $\left. \frac{d}{dt} \right|_{t=0^+} \xt(\xb) = \di \Vb(\xb)$.
		Corollary \ref{cor:frechet_diff} combined with the product rule of Fr\'{e}chet derivatives yields the Fr\'{e}chet derivative 
		\begin{align*}
			&\left. \frac{d}{dt} \right|_{t=0^+} \left( \left(\advar^0(\xb) - \advar^0(\yb)\right) \left(\weakSol^0(\xb) - \weakSol^0(\yb) \right) \kernelt(\xb,\yb)\xt(\xb)\xt(\xb) \right) \\
			&= \left(\advar^0(\xb) - \advar^0(\yb)\right) \left(\weakSol^0(\xb) - \weakSol^0(\yb) \right)\left(\grad_{\xb} \kernel(\xb, \yb)^\top \Vb(\xb) + \grad_{\yb} \kernel(\xb, \yb)^\top \Vb(\yb)\right) \\
			&+ \left(\advar^0(\xb) - \advar^0(\yb)\right) \left(\weakSol^0(\xb) - \weakSol^0(\yb) \right) \kernel(\xb,\yb) \left(\di \Vb(\xb) + \di \Vb(\yb)\right) \text{ in } L^1(\nlDom_{nl} \times \left(\nlDom_{nl} \cup \nlBound_{nl} \right), \R).   
		\end{align*}
		Moreover, according to, e.g. \cite[Lemma 2.14]{Sturm_diss} 
		\begin{align*}
			\left. \frac{d}{dt} \right|_{t=s} \transformationMatrix(t)(\xb) =& \tr \left( D\Vb(\xb) \left(D\Ftb(\xb)\right)^{-1} \right)\transformationMatrix(t)(\xb) \\
			&- \left(D\Ftb(\xb)\right)^{-1}D\Vb(\xb)\transformationMatrix(t)(\xb) - \left( \left(D\Ftb(\xb)\right)^{-1}D\Vb(\xb)\transformationMatrix(t)(\xb)\right)^\top,
		\end{align*}
		and therefore $\left. \frac{d}{dt} \right|_{t=0^+} \transformationMatrix(t) = \di \Vb - D\Vb - D \Vb^\top$. Finally, we derive by again utilizing the product rule of Fr\'{e}chet derivatives
		\begin{align*}
			&\left.\frac{d}{dt}\right|_{t=0^+} \left( \transformationMatrix(t)(\xb) \grad \weakSol^0(\xb), \grad \advar^0(\xb) \right) \\
			&= - \left( \left( \grad \Vb(\xb) + \grad \Vb(\xb)^\top \right) \grad \weakSol^0(\xb), \grad \advar^0(\xb) \right) + \left( \grad \weakSol^0(\xb), \grad \advar^0(\xb) \right) \di \Vb(\xb) \text{ in } L^1(\nlDom_l, \R).
		\end{align*}
	\end{proof}
	Combining Theorems \ref{theo:shape_der_1} and \ref{theo:shape_der_2} results in a formula for the shape derivative of the reduced objective functional $D_{\shape}\objFunc^{red}(\shape)[\Vb]$.
	\begin{corollary}
		Let the reduced objective functional $\objFunc$ be defined as in \eqref{eq_reduced_functional} and $\kernel \in W^{1,\infty}(\left(\completeDom\right) \times \left(\completeDom\right))$ is supposed to satisfy (K1)-(K4).\\
		Then, for $\Vb \in \vecfieldsspecific$ we can express the shape derivative $D_{\shape}\objFunc^{red}(\shape)[\Vb]$ as
		\begin{align*}
			&D_{\shape}\objFunc^{red}(\shape)[\Vb] = D_{\shape} \objFunc(\weakSol^0, \shape)[\Vb] + D_{\shape} \interfaceOp(\weakSol^0, \advar^0)[\Vb] - D_{\shape} \interfaceForce(\advar^0)[\Vb] \\
			&= \int_{\nlDom} - \left( \weakSol^0 - \data \right) \grad \data^{\top} \Vb + \left(\weakSol^0 - \data \right)^2 \di \Vb ~d\xb + \nu \int_{\shape} \di \Vb - \nb^{\top}\grad \Vb^\top \nb ~ds \\
			&+ \int\limits_{\nlDom_l} - \left( \left( \grad \Vb + \grad \Vb^\top \right) \grad \weakSol^0(\xb), \grad \advar^0(\xb) \right) + \left( \grad \weakSol^0(\xb), \grad \advar^0(\xb) \right) \di \Vb(\xb) ~d\xb \\
			&+ \int\limits_{\nlDom_{nl}} \int\limits_{\nlDom_{nl} \cup \nlBound_{nl}} \left(\advar^0(\xb) - \advar^0(\yb)\right) \left(\weakSol^0(\xb) - \weakSol^0(\yb)\right)\left(\grad_{\xb} \kernel(\xb,\yb)^\top \Vb(\xb) + \grad_{\yb} \kernel(\xb,\yb)^\top \Vb(\yb) \right) ~d\yb d\xb \\
			&+ \int\limits_{\nlDom_{nl}} \int\limits_{\nlDom_{nl} \cup \nlBound_{nl}} \left(\advar^0(\xb) - \advar^0(\yb)\right) \left(\weakSol^0(\xb) - \weakSol^0(\yb)\right)\left(\di \Vb(\xb) + \di \Vb(\yb) \right) ~d\yb d\xb \\
			&- \int_{\nlDom} \left(\grad f_{\shape}^{\top} \Vb \right) \advar^0 + f_{\shape} \advar^0 \di \Vb ~d\xb.
		\end{align*}
	\end{corollary}
	\subsection{Optimization Algorithm}
	\label{section:shape_opt_algorithm}
	In this section we outline the optimization procedure to solve the interface identification problem constrained by Local-to-Nonlocal coupling as described in Section \ref{section:interface_problem_formulation}. 
	As shown in \cite{linear_view}, shape optimization techniques can be identified with corresponding methods in the Hilbert space $\left(\Hc, g\right)$, where $\Hc \subset \{T|T:\nlDom \rightarrow \Rd\}$ and $g$ is an inner product on $\Hc$. Therefore, shape optimization problems can be investigated as optimization problems in $\left(\Hc, g\right)$, see \cite{linear_view}.\\
	In the optimization algorithm, that we will deploy, we would like to derive a Riesz-representation of the shape gradient $\grad \redFunc$ for a scalar product $\shapeInnerProd_{\shape}$ that can depend on the current shape $\shape$. Then, given the state $\weakSol^0$ and adjoint solution $\advar^0$, we compute $\grad \redFunc$ by solving
	\begin{align}
		\label{shape_gradient}
		\shapeInnerProd_{\shape}(\grad \redFunc, \Vb) = D_\shape \redFunc(\shape)[\Vb] \quad \text{for all } \Vb \in \vecfieldsspecific.
	\end{align} 
	In our case, we choose the linear elasticity operator as an inner product, i.e.  
	\begin{align*}
		&\shapeInnerProd_{\shape}: \vecfieldsspecific \times \vecfieldsspecific \rightarrow \R,\quad g(\Ub, \Vb) \defas \int_{\nlDom} \sigma(\Ub):\epsilon(\Vb) ~d\xb, \text{ where}\\
		&\sigma(\Ub) \defas \lambda tr(\epsilon(\Ub))\Id +2\mu\epsilon(\Ub) \text{ and } \epsilon(\Ub) \defas \frac{1}{2} \left( \grad \Ub + \grad \Ub^\top \right).
	\end{align*} 
	Here, $\sigma(\Ub)$ is called the strain and $\epsilon(\Ub)$ stress tensor of $\Ub$ with Lame parameters $\lambda$ and $\mu$. According to \cite{schulz2016computational} a locally varying choice for $\mu$ can have a stabilizing effect on the mesh. So, we follow \cite{schulz2016computational} and set $\lambda=0$ and $\mu(\shape) \defas \mu$ as the solution to
	\begin{alignat}{2}
		\Delta \mu &= 0 \quad && \text{on } \nlDom, \nonumber \\
		\mu &= \mu_{min} \quad && \text{on } \partial \nlDom,\label{mu_heuristic} \\
		\mu &= \mu_{max} \quad && \text{on } \shape.\nonumber
	\end{alignat}
	As a consequence $\mu(\shape)$ is dependent on the current shape $\shape$ and therefore the inner product $\shapeInnerProd_{\shape}$ is also depending on the interface $\shape$. Moreover, the choice of $\mu_{min}$, $\mu_{max} \in \R$ has a strong effect on the step size and the speed of convergence of the optimization algorithm.\\
	Further, we apply the finite element method to solve the described interface identification problem. Accordingly, the domain $\nlDom$ is represented by a mesh and the boundary $\shape$ is characterized by edges. The triangle vertices and, as a consequence, also the edges are then deformed by a vector $\Ub_k$, which resembles the finite element approximation of the descent direction in the linear space of deformations. Here, $\Ub_k$ is a linear combination of piecewise linear and continuous basis functions, i.e. we apply continuous Galerkin. On the other hand, the state and adjoint solution are chosen to be constructed by a linear combination of piecewise linear and continuous basis functions on $\nlDom_{nl}$ and on $\nlDom_l$. As a result, the state and adjoint can be discontinuous on the interface $\shape$.\\
	Furthermore, for vector fields that do not deform the interface $\shape$, i.e. $\supp(\Vb) \cap \shape = \emptyset$, it should hold that $D_{\shape} \redFunc(\shape)[\Vb] = 0$. However, in the finite element setting it can happen that for such test vector fields $\Vb$ the corresponding shape derivative $D_{\shape}\redFunc[\Vb]$ is not zero due to discretization errors. Following \cite{schulz2016efficient}, we therefore set in every iteration
	\begin{align*}
		D_{\shape}\redFunc(\shape)[\Vb] = 0 \text{ for } \Vb \in \vecfieldsspecific \text{ with } \supp(\Vb) \cap \shape = \emptyset.
	\end{align*}
	The complete method is shown in Algorithm \ref{alg:shape_opt}.\\
	\begin{algorithm}
		\label{alg:shape_opt}
		Given: $\data$, $\kernel$, $f_{\shape}$, $\texttt{maxiter} \in \N$, $\texttt{tol} \in (0,\infty)$.\\
		\While{$k \leq \texttt{maxiter}$ and $||D_{\shape} \redFunc(\shape_k)|| > \texttt{tol}$}{
			Interpolate data $\data$ onto the current finite element mesh $\nlDom_k$\\
			$\weakSol(\shape_k) \leftarrow$ solve state equation w.r.t. $\shape_k$.\\
			$\advar(\shape_k) \leftarrow$ Solve adjoint equation w.r.t. $\shape_k$.\\
			$D_\shape \redFunc(\shape_k)[\Vb] \leftarrow$ Assemble the shape derivative for all basis functions $\Vb$.\\
			Set $D_\shape \redFunc(\shape_k)[\Vb]=0$ if $\supp (\Vb) \cap \shape_k = \emptyset$.\\
			$\mu \leftarrow$ Derive the locally varying Lam\'{e} parameter.\\
			$\grad \redFunc(\shape_k) \leftarrow$ Solve \eqref{shape_gradient} to get Riesz-representation of the shape gradient.\\
			\eIf{curvature condition fulfilled}{
				$\Ub_k$ = L-BFGS-Update
			}{
				$\Ub_k = - \grad \redFunc(\shape_k)$
			}
			\While{$\redFunc((\Id + \alpha_k\Ub_k)(\shape_k)) > \redFunc(\shape_k) + c D_{\shape} \redFunc(\shape_k)[\Ub_k]$\label{shape_algo:backtracking_start}} 
			{$\alpha = \tau \alpha$\label{shape_algo:backtracking_end}}
			$\alpha_k \leftarrow \alpha$\\
			$\nlDom_{k+1} = \left(\Id + \alpha_k \Ub_k\right)(\nlDom_k)$\\
			$k=k+1$
		}
		\caption{Shape optimization algorithm}
	\end{algorithm}
	Additionally, after we have found a shape gradient $\grad \redFunc(\shape_k)$ we make use of an limited memory BFGS-update, if the corresponding curvature condition is satisfied. Otherwise we take the negative gradient as the descent direction. For more information on how to apply L-BFGS, we refer to \cite{num_opt}.
	After we have found a descent direction $\Ub_k$ we compute the step length by applying a backtracking line search in  Lines \ref{shape_algo:backtracking_start}-\ref{shape_algo:backtracking_end}, where we stop after the sufficient decrease condition is satisfied. The parameter $c$ is chosen to be $c=10^{-4}$ as it is suggested in \cite{num_opt}.
	
	\begin{remark}
		Since we use an iterative solver to compute solutions to the state and adjoint equation, the application of (multi-step) one-shot methods as in \cite{one-shot,multistep}, where state, adjoint and the shape are simultaneously calculated, seems to be a natural way to solve the interface identification problem in a faster time. However, the assembly of the stiffness matrix needed to compute the state or adjoint solution is quite costly and compared to that the time saved by stopping the algorithm for calculating $\weakSol(\shape_k)$ and $\advar(\shape_k)$ after a few iterations is rather small. Consequently, if we need more outer iteration steps, (multi-step) one-shot methods do not yield a faster way to solve the interface identification problem constrained by a LtN coupling. On the other hand, we can only reduce the time by a few seconds, if we slightly increase the tolerance for solving the state and adjoint equation such that we still need the same number of outer iterations. Therefore, we keep the standard approach introduced in this Section and do not use one-shot methods.
	\end{remark}
	
	\section{Numerical Experiments}
	\label{section:num_ex}
	In this section we demonstrate the shape optimization algorithm that was introduced in the previous Section \ref{section:shape_opt_algorithm} on two examples. Snapshots of the algorithm are depicted in Figure \ref{fig:ex_1} for the first and in Figure \ref{fig:ex_2} for the second example.
	\begin{figure}[h!] 
		\begin{center}
			\begin{small}
				\begin{tabular}{cccc}
					\includegraphics[width = 0.2\textwidth]{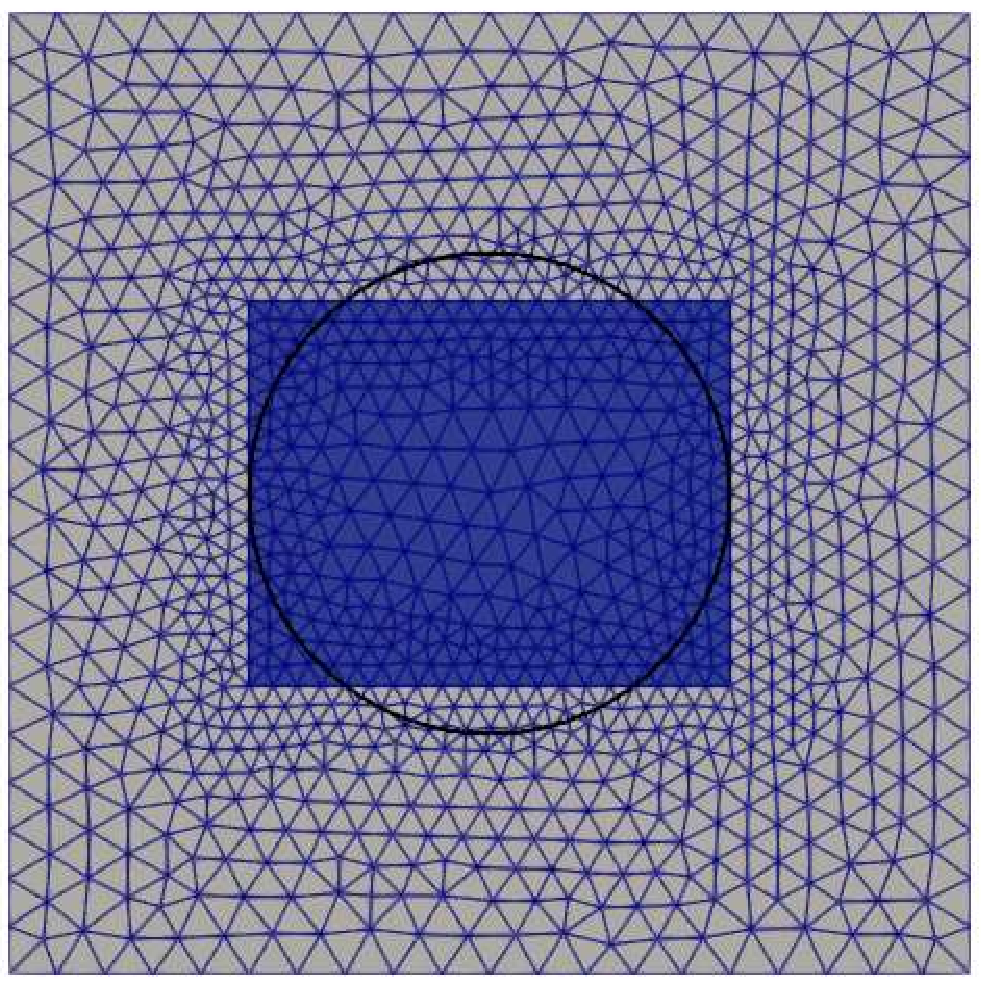}
					&\includegraphics[width = 0.2\textwidth]{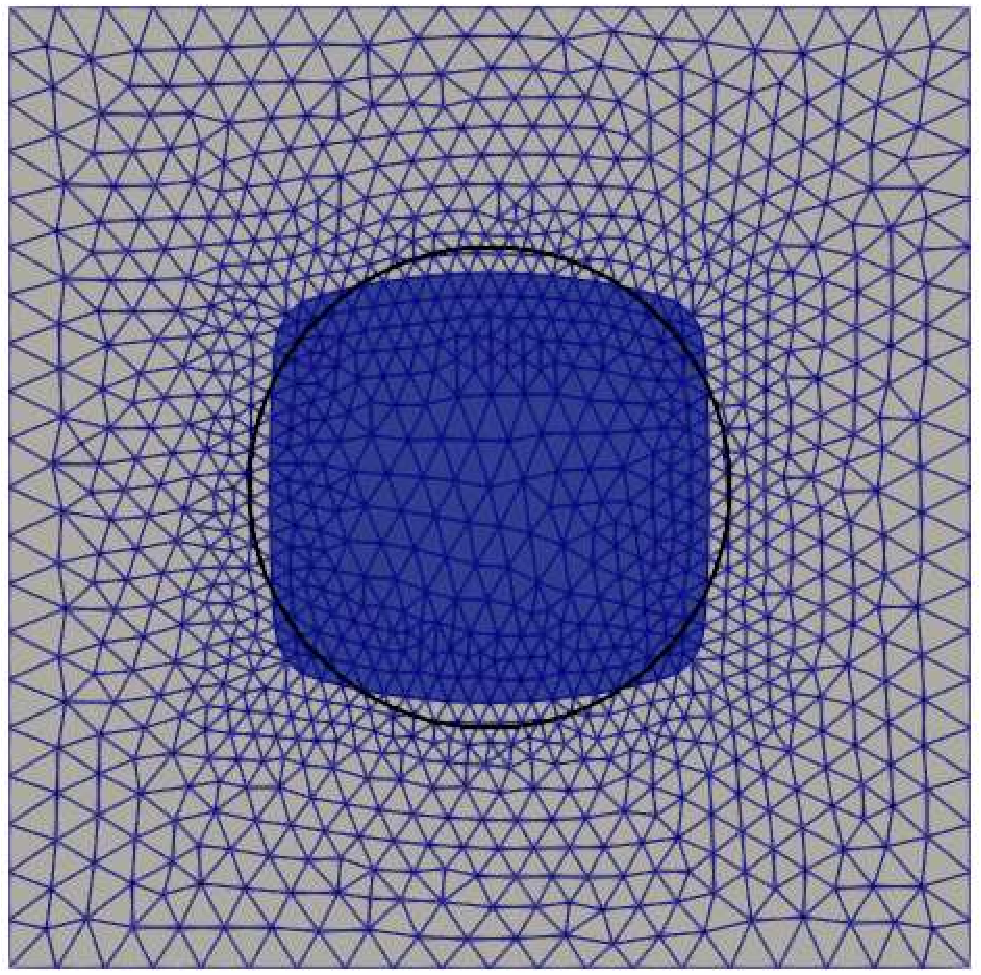}
					&\includegraphics[width = 0.2\textwidth]{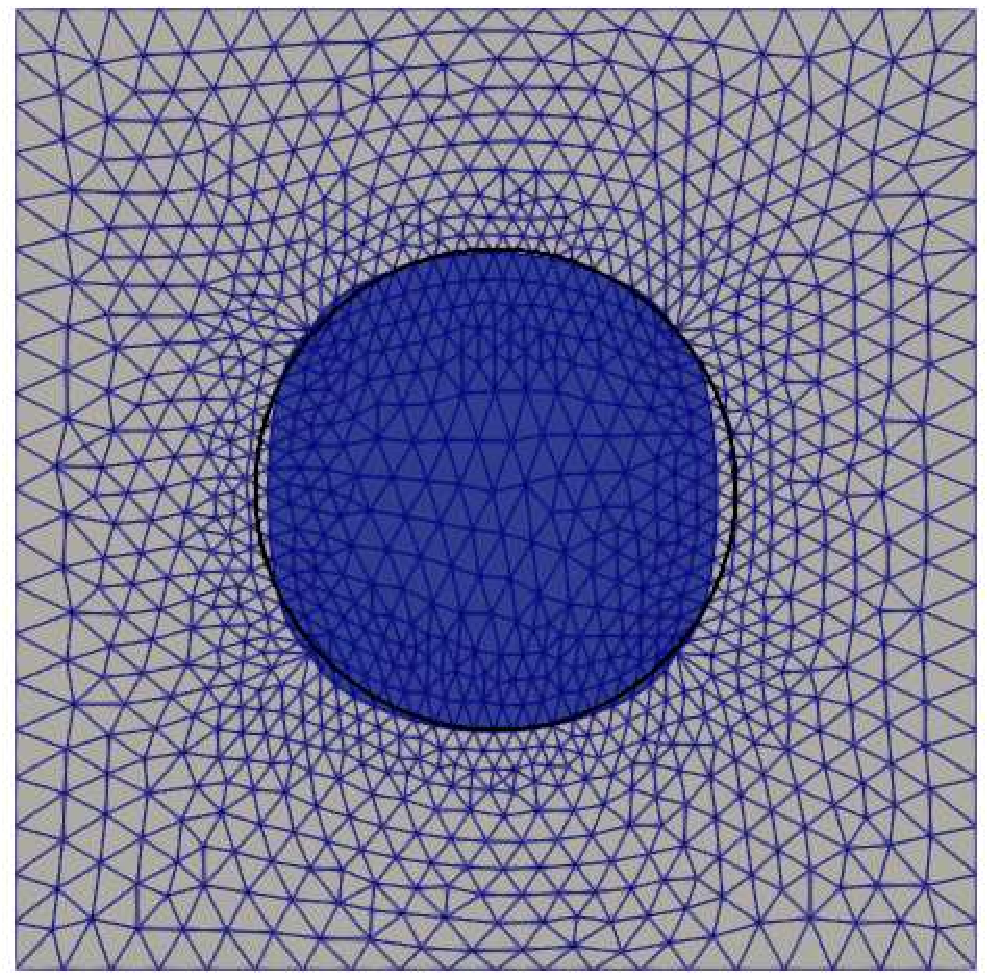}
					&\includegraphics[width = 0.2\textwidth]{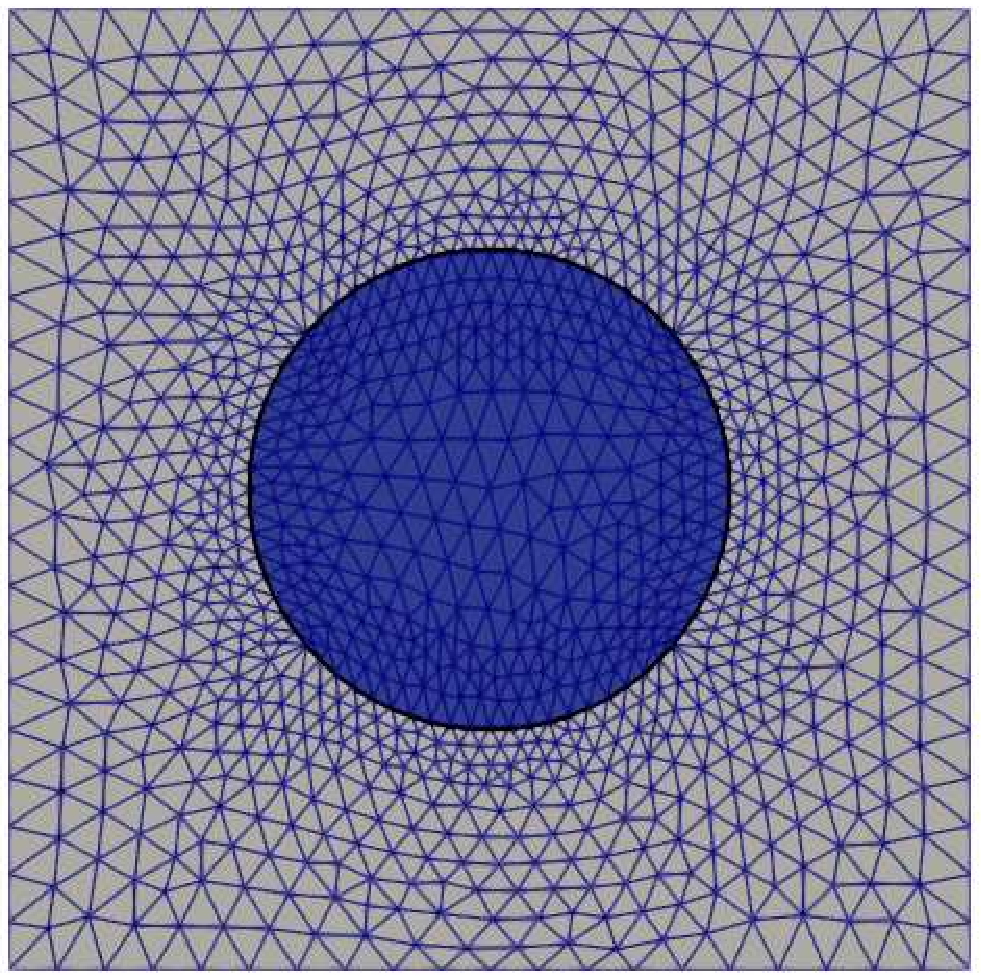}\\
					Start setup & Iteration 3 & Iteration 5 & Iteration 14
				\end{tabular}
			\end{small}
		\end{center}
		\caption{Example 1}
		\label{fig:ex_1}
	\end{figure}
	\begin{figure}[h!] 
		\begin{center}
			\begin{small}
				\begin{tabular}{cccc}
					\includegraphics[width = 0.2\textwidth]{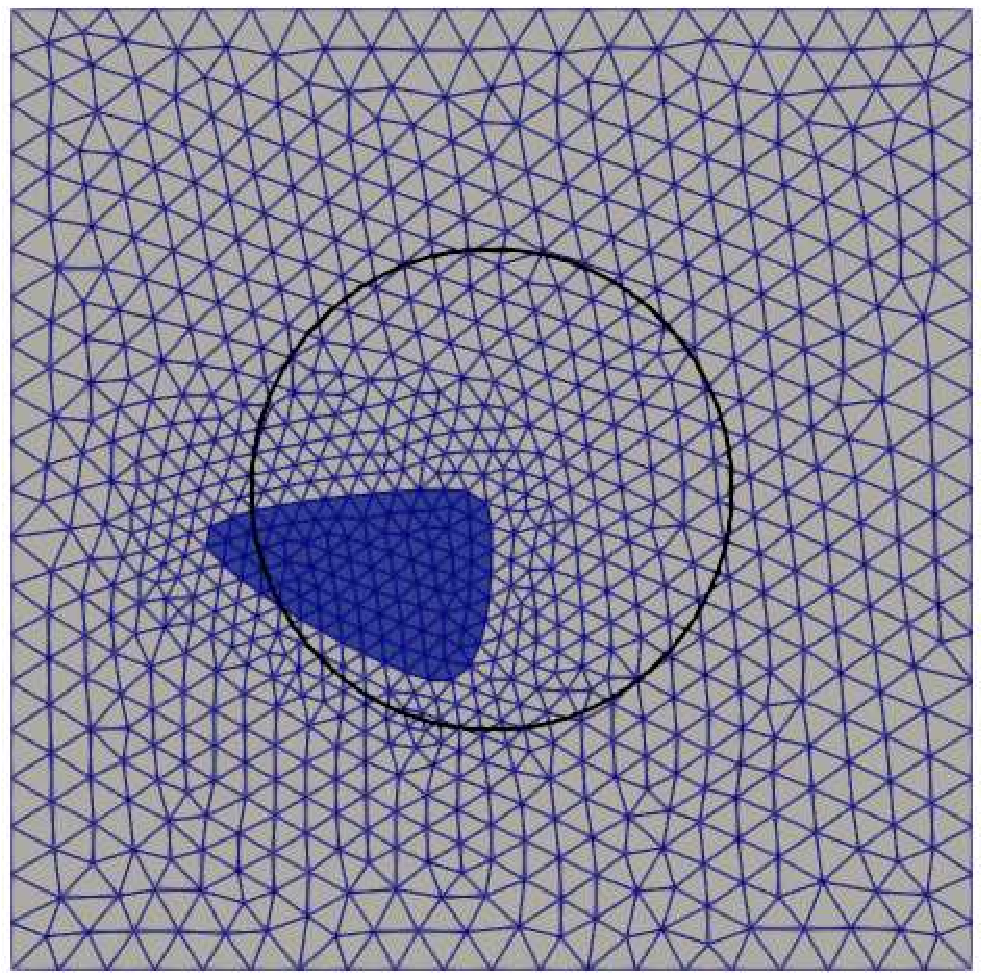}
					&\includegraphics[width = 0.2\textwidth]{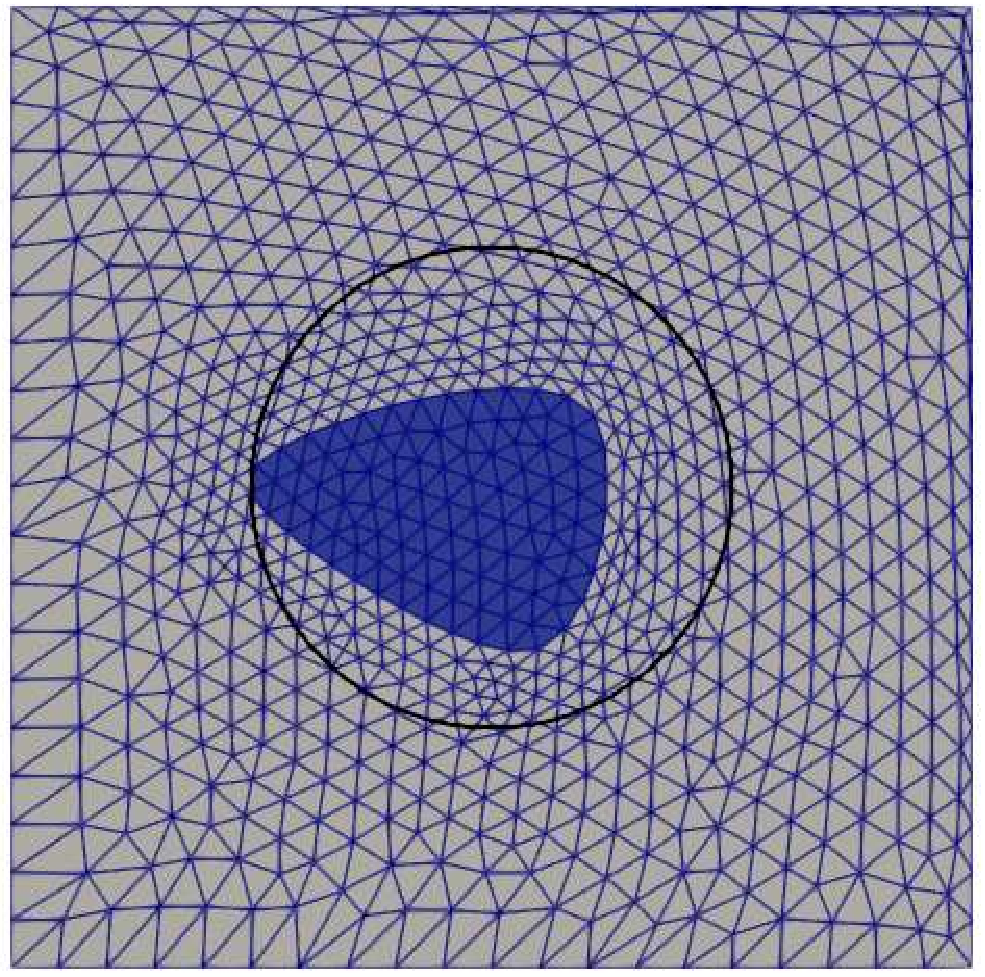}
					&\includegraphics[width = 0.2\textwidth]{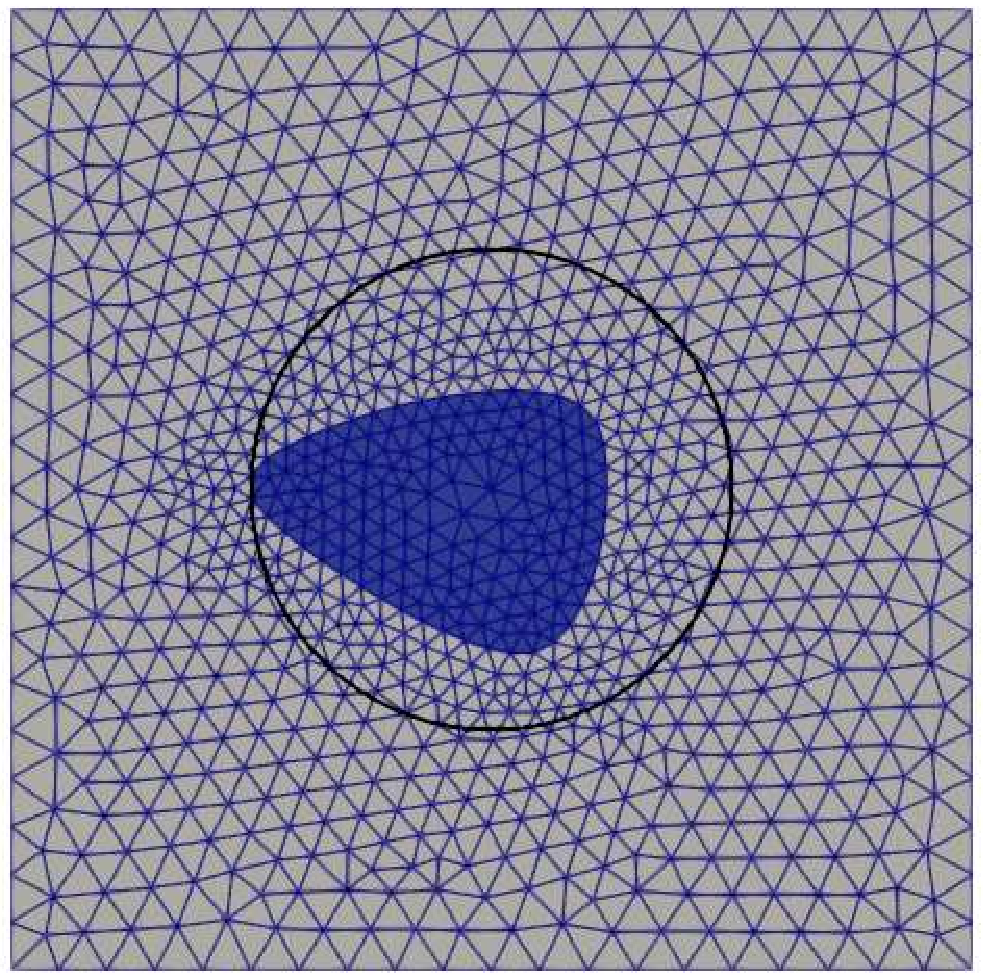}
					&\includegraphics[width = 0.2\textwidth]{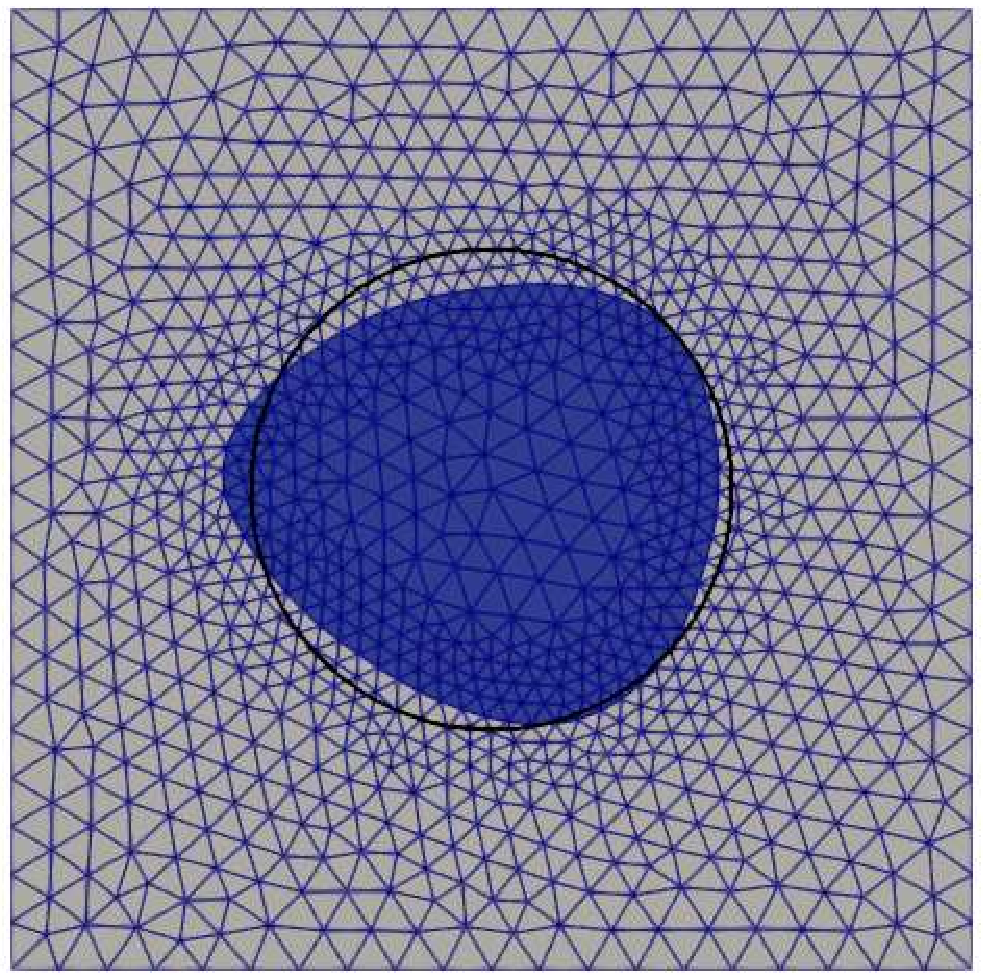}\\
					Start setup & Iteration 5 & Iteration 5 remeshed & Iteration 10 remeshed\\
					&\includegraphics[width = 0.2\textwidth]{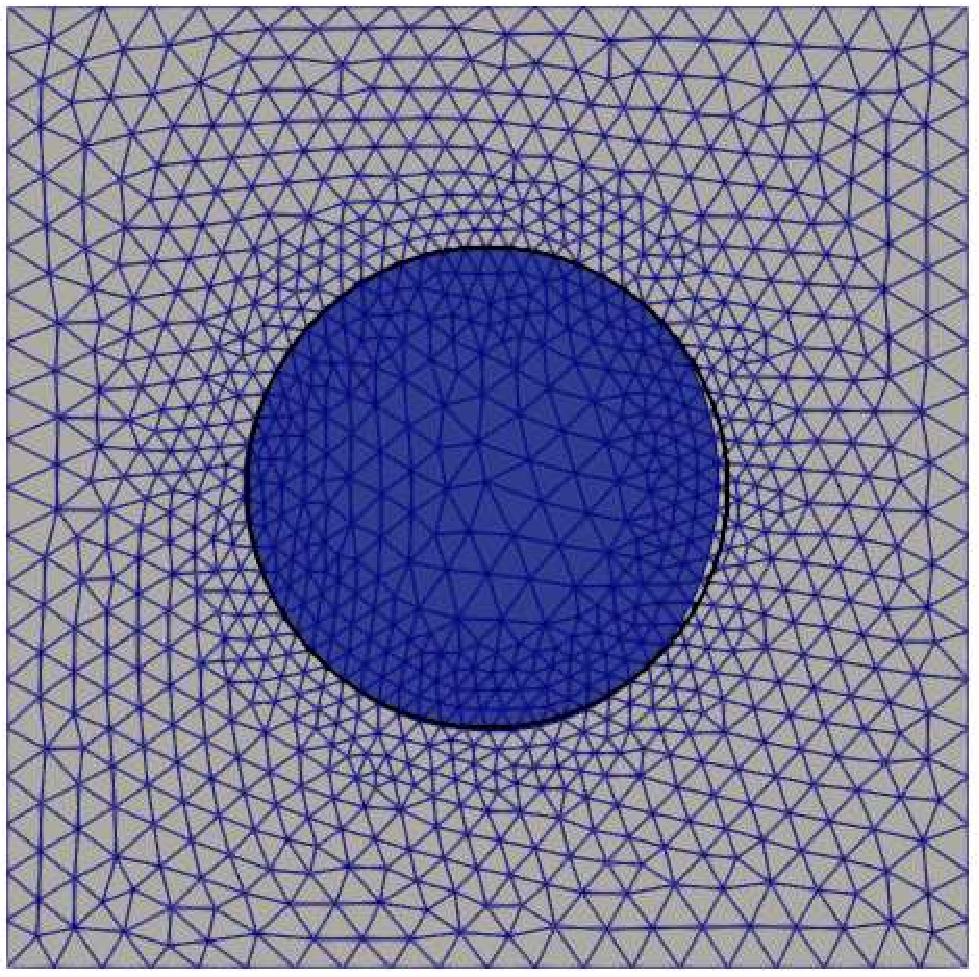}
					&\includegraphics[width = 0.2\textwidth]{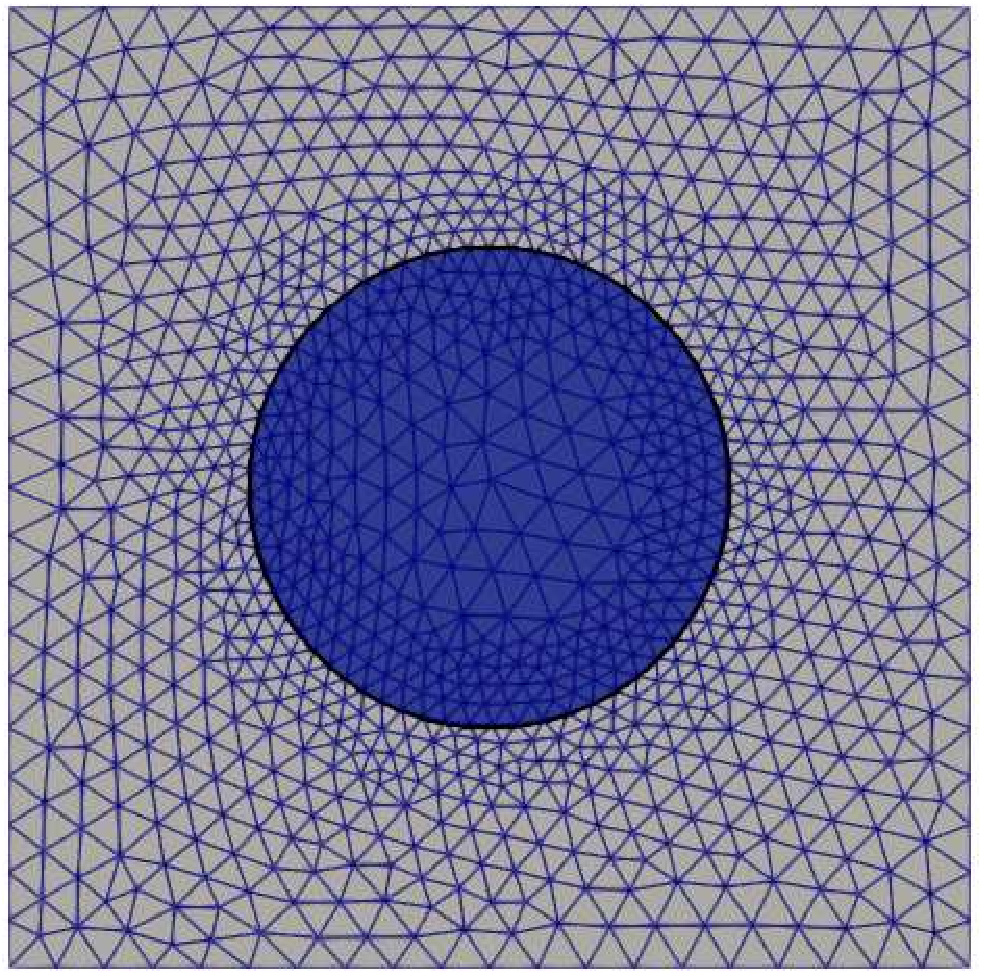}
					&\\
					& Iteration 15 & Iteration 20 & 
				\end{tabular}
			\end{small}
		\end{center}
		\caption{Example 2}
		\label{fig:ex_2}
	\end{figure}
	In both cases, we choose a small perimeter regularization with parameter $\nu = 0.001$ and the forcing term
	\begin{align*}
		f_{\shape}(\xb) = 10 \ind_{\nlDom_{nl}}(\xb) - 10 \ind_{\nlDom_l}(\xb).
	\end{align*}
	In the first experiment, we use the kernel $\kernel_1$ and in the second example the kernel $\kernel_2$, which are defined as follows 
	\begin{align*}
		\kernel_1(\xb,\yb) = \frac{4}{\pi \delta^4}\ind_{B_{\delta}(\xb)}(\yb) \quad \text{and} \quad \kernel_2(\xb,\yb) = \frac{4}{\pi \delta^4}\left( 1 - 0.5 \frac{||\xb - \yb||^2}{\delta^2} \right)\ind_{B_{\delta}(\xb)}(\yb). 
	\end{align*}
	Since in the first case the kernel $\kernel_1$  is piecewise constant, the shape derivative of the corresponding LtN-operator \eqref{eq:shape_der_LtN_op} reduces to
	\begin{align*}
		D_{\shape} &\interfaceOp(\weakSol^0,\advar^0)[\Vb] = \int\limits_{\nlDom_l} - \left( \left( \grad \Vb + \grad \Vb^\top \right) \grad \weakSol^0(\xb), \grad \advar^0(\xb) \right) + \left( \grad \weakSol^0(\xb), \grad \advar^0(\xb) \right) \di \Vb(\xb) ~d\xb \\
		&+ \int\limits_{\nlDom_{nl}} \int\limits_{\nlDom_{nl} \cup \nlBound_{nl}} \left(\advar^0(\xb) - \advar^0(\yb)\right) \left(\weakSol^0(\xb) - \weakSol^0(\yb)\right)\left(\di \Vb(\xb) + \di \Vb(\yb) \right) ~d\yb d\xb.
	\end{align*}
	Now, we solve the interface identification problem described in Section \ref{section:interface_problem_formulation} by applying the finite element method, where we choose to employ piecewise linear and continuous basis functions. 
	Here, the underlying mesh is generated by Gmsh \cite{gmsh}.
	The nonlocal components for the stiffness matrix regarding $\interfaceOp$ that is used in the Schwarz method, see Section \ref{section:schwarz}, as well as the shape derivative $D_{\shape} \interfaceOp$ are computed by a modified version of the Python package nlfem \cite{nlfem}. The local components and the shape derivatives of the objective functional $D_{\shape} \objFunc$ and of the forcing term $D_{\shape}  \interfaceForce$ are assembled by using FEniCS \cite{FEniCS1, FEniCS2}.
	Then, the data $\data$ is computed as the solution of the Local-to-Nonlocal coupling regarding a circle as the interface $\shape$, i.e. the boundary of the nonlocal domain. This circle can be observed in black on every picture. After that, we start with a different interface $\shape_0$ and try to approximate $\data$ by following Algorithm \ref{alg:shape_opt}, where we stop, if $||D_{\shape}\redFunc(\shape_k)|| < 5 \cdot 10^{-5}$ or if $k > 25$. As a result, the nonlocal domain $\nlDom_{nl}$, which is colored in blue, should roughly end up as the disc with the black circle as the boundary.\\
	In the first example the algorithm terminates after 14 iterations. Here, the edges of the initial shape $\shape_0$ disappear due to the perimeter regularization and the nonlocal domain $\nlDom_{nl}$ has approximately the above described target shape. As mentioned in Section \ref{section:shape_opt_algorithm} the convergence speed and therefore the number of iterations is highly dependent on the choice of Lam\'{e} parameters, which we derive from \eqref{mu_heuristic} with $\mu_{min}=0.0$ and $\mu_{max}=1.0$ for both experiments.\\
	In the second example, our initial shape is small compared to the target shape and needs to move a bit to the upper right and therefore the mesh quality decreases. Also, nodes in $\nlDom$ are not allowed to be pushed outside of $\nlDom$, which is tested in the code by a function. So, in the beginning of the algorithm some nodes are moved towards the boundary and as a result the algorithm stagnates, because they cannot be pushed further. Therefore, we remesh after the fifth and tenth iteration. In the end, the algorithm stops after 20 iterations and we are able to recover the target shape quite well, which can be seen in Picture \ref{fig:ex_2}.\\
	The development of the objective function values and the norm of the shape gradients for both examples are presented in Figure \ref{fig:comparison_objective_funtional}. 
	\begin{figure}[h!]
		\includegraphics[width = 0.5\textwidth]{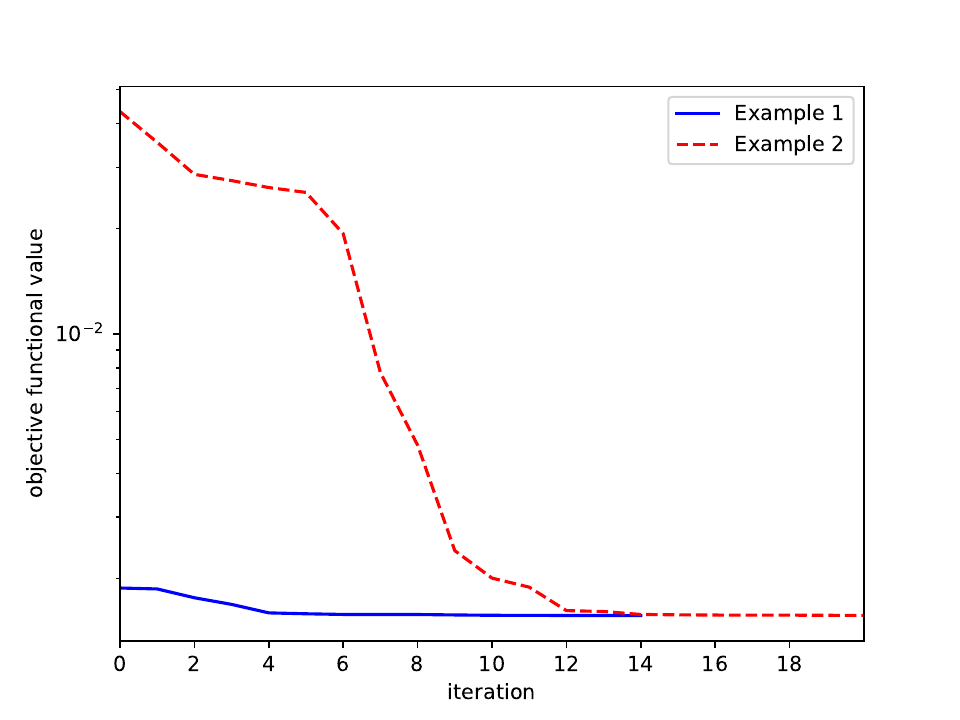}
		\includegraphics[width = 0.5\textwidth]{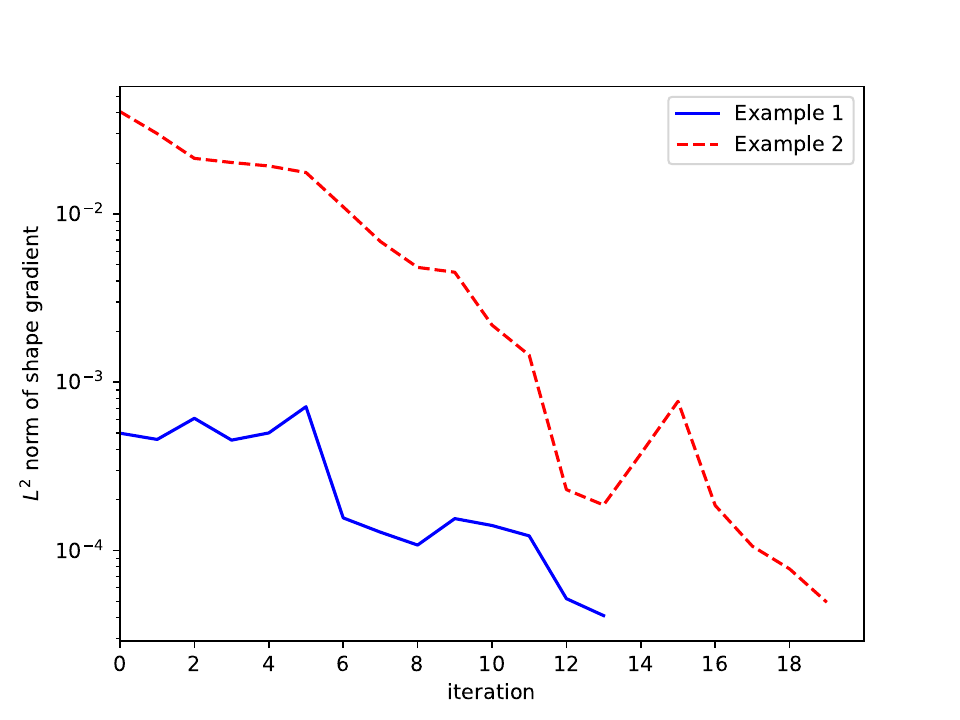}
		\caption{On the left hand side we see the development of the objective functional values during the algorithm for each example. Here, the starting solution of the first example is quite near the optimal solution such that the objective function value decreases only a bit. In the second example the objective function value is reduced quite fast in the first twelve iterations. After that there is also only a small improvement regarding the objective function value. On the right hand side, we can see the history of the $L^2(\nlDom)$ norm of the shape gradients.}
		\label{fig:comparison_objective_funtional}
	\end{figure}
	\section{Conclusion}
	As we have seen in this work, shape optimization techniques can be applied on interface identification problems that are constrained by an energy-based Local-to-Nonlocal coupling in order to find out where the nonlocal domain of this coupling should be located. Here, the shape derivative of the reduced functional is derived by the averaged adjoint method. In the future it could be interesting to enlarge this framework to applications in the field of Peridynamics, where most of these LtN combinations are currently employed.

	\section*{Acknowledgements}
	
	This work has been supported by the German Research Foundation (DFG) within the Research Training Group 2126: 'Algorithmic Optimization'.

	\section*{Disclosure statement}
	
	The authors report there are no competing interests to declare.

	\bibliographystyle{plain}
	\bibliography{literature}
	
	\appendix
	\section{Proof of Lemma \ref{lemma:AAM_reduced_functional}}
	For this proof we rely on the next two Lemmata:
	\begin{lemma}
		\label{lemma:poincare_LtN}
		There exists a positive constant $C \in (0, \infty)$, such that 
		\begin{align*}
			\LtNOp(t, \weakSol, \weakSol) \geq  C ||\weakSol||_{H^1(\nlDom_l) \times L_c^2(\nlDom_{nl} \cup \nlBound_{nl})}^2 \quad \text{for all } \weakSol \in \LtNSpace(\shape) \text{ and } t \in [0, T].
		\end{align*}
	\end{lemma}
	\begin{proof}
		Since there exists a positive constant 
		$0 < C_1 < \infty$ with $C_1 |\xb|^2 \leq \xb^\top \transformationMatrix(t)(\xb)$(see \cite[Chapter 10.2.4]{shapes_geometries}), if $t$ is small enough, and since $\kernel(\xb,\yb)$ is bounded from below by $\kernel_0$, if $\yb \in B_{\epsilon}(\xb)$(see Chapter \ref{chap:Dir_Prob}), we set $C_2 \defas \min\{C_1, \kernel_0\}$ and notice that
		\begin{align*}
			\LtNOp(t, \weakSol, \weakSol) =& \int_{\nlDom_l} \left( \transformationMatrix(t)\grad \weakSol , \grad \weakSol \right) ~d\xb + \int_{\nlDom_{nl}} \int_{\nlDom_{nl} \cup \nlDom_l} \left( \weakSol(\xb) - \weakSol(\yb) \right)^2 \kernelt(\xb,\yb)\xt(\xb) \xt(\yb) ~d\yb d\xb \\ &+ \int_{\nlDom_{nl}} \weakSol(\xb)^2 \int_{\nlBound} \kernelt(\xb,\yb) \xt(\xb) ~d\yb d\xb \\
			\geq& C_2 \int_{\nlDom_l} \grad \weakSol^2 ~d\xb + \int_{\nlDom_{nl}} \int_{\nlDom_{nl} \cup \nlDom_l} \left( \weakSol(\xb) - \weakSol(\yb) \right)^2 C_2 \ballxy ~d\yb d\xb \\ 
			&+ \int_{\nlDom_{nl}} \weakSol(\xb)^2 \int_{\nlBound} C_2 \ballxy ~d\yb d\xb \geq C||\weakSol||_{H^1(\nlDom_l) \times L_c^2(\nlDom_{nl} \cup \nlBound_{nl} )}^2,
		\end{align*} 
		where the last step is a result from Lemma \ref{lemma:AcostaNormEquivalence} applied on $\tilde{\kernel}(\xb,\yb) = C_2 \ballxy$ in the corresponding bilinear form, which yields such a constant $C \in (0, \infty)$.
	\end{proof}
	\begin{lemma}[After {\cite[Chapter 10.2.4 Lemma 2.1]{shapes_geometries}}]
		\label{lemma:l2_convergence}
		Let $\hat{\nlDom} \subset \Rd$ be a bounded and open domain with nonzero measure, $\Vb \in C_0^1(\hat{\nlDom},\Rd)$ and $n \in \N$. Then, for $g \in L^2(\hat{\nlDom},\R^n)$, we get
		\begin{align*}
			||g \circ \Ftb - g||_{L^2(\hat{\nlDom},\R^n)} \rightarrow 0 \text{ and } ||g \circ \Ftb^{-1} - g||_{L^2(\hat{\nlDom},\R^n)} \rightarrow 0 \quad \text{for } t \searrow 0.
		\end{align*}
	\end{lemma} 
	\begin{proof}
		Since $g$ and $\Vb$ can be extended by zero to functions in $L^2(\Rd)$ and $C_0^1(\Rd,\Rd)$, respectively, we refer for the case $n=1$ to the proof of \cite[Chapter 10.2.4 Lemma 2.1]{shapes_geometries}. Then, the cases $n \in \N$ with $n \geq 2$ are a direct consequence. 
	\end{proof}
	\label{appendix:proof_AAM}
	\begin{proof}
		Define $f_\shapet(\xb) \defas f_\shape(F_t(\xb))$
		and ${\nabla f_\shapet(\xb) \defas \nabla f_\shape(F_t(\xb))}$.\\
		\textbf{Assumption(H0):}\\
		We prove that $\AAMLagFunc(t,su^t + (1-s)u^0,\advar)$ is absolutely continuous in $s$ by showing that $\tilde{\AAMLagFunc}:[0,1]\rightarrow\R, \tilde{\AAMLagFunc}(s)=\AAMLagFunc(t,s\weakSol^t + (1-s)\weakSol^0,\advar)$ is continuously differentiable.
		Since $\varOp$ is linear in the second argument, we can directly conclude
		\begin{align*}
			\partial_s \tilde{\AAMLagFunc}(s) &= d_u\AAMLagFunc(t,s\weakSol^t + (1-s)\weakSol^0,\advar)[\weakSol^t - \weakSol^0]\\ 
			&= d_u \LtNOp(t, s\weakSol^t + (1-s)\weakSol^0, \advar)[\weakSol^t - \weakSol^0] + d_u \objFunc(t, s\weakSol^t + (1-s)\weakSol^0)[\weakSol^t - \weakSol^0]\\
			&= \LtNOp(t,\weakSol^t - \weakSol^0,\advar) + \int_{\nlDom} \left(s\weakSol^t + (1-s)\weakSol^0 - \data^t \right)\left(\weakSol^t - \weakSol^0\right)\xt ~d\xb \\
			&= \LtNOp(t,\weakSol^t - \weakSol^0,\advar) + s \int_{\nlDom} \left(\weakSol^t - \weakSol^0\right)^2\xt ~d\xb + \int_{\nlDom} \left(\weakSol^0 - \data^t \right) \left( \weakSol^t - \weakSol^0 \right)\xt ~d\xb,
		\end{align*}
		which is clearly continuous in $s$, i.e. $\tilde{\AAMLagFunc}$ is continuously differentiable and the first part of Assumption (H0) holds.
		Furthermore, the second criterion of (H0) is also fulfilled, since
		\begin{align*}
			&\int_0^1 \left| d_u \AAMLagFunc(t,su^t +  (1-s)u^0,\advar)[\tilde{u}] \right| ~ds = \int_0^1 \left| \LtNOp(t,\tilde{u},\advar) + \int_{\nlDom} \left( s(u^t-u^0) + (u^0 - \data^t) \right)\tilde{u}\xt ~d\xb \right| ~ds \\
			&\leq \left| \LtNOp(t, \tilde{u}, \advar) \right| + \frac{1}{2} \left| \int_{\nlDom} \left( u^t - u^0 \right) \tilde{u}\xt ~d\xb \right| + \left| \int_{\nlDom} \left( u^0 - \data^t \right)\tilde{u}\xt ~d\xb \right| < \infty.
		\end{align*}
		Notice that the averaged adjoint equation \eqref{AAE} can be rephrased as
		\begin{align}
			\label{LtNAAE}
			\LtNOp(t,\tilde{u},\advar^t) = - \int_{\nlDom} \left( \frac{1}{2}(u^t + u^0) - \data^t \right) \tilde{u} \xt ~d\xb,
		\end{align}
		where the right-hand side can be interpreted as a linear and continuous operator with regards to $\tilde{u}$ and therefore, equation \eqref{LtNAAE} is  also a Local-to-Nonlocal coupling problem which has a unique solution $\advar^t \in \LtNSpace(\shape)$.\\
		We now show that $\{u^t\}_{t\in[0,T]}$ and $\{\advar^t\}_{t\in[0,T]}$ are bounded, which will be used to prove assumption (H1).  
		Therefore Lemma \ref{lemma:poincare_LtN} yields the existence of a constant $C \in (0, \infty)$ with
		\begin{align}
			&||u^t||_{H^1(\nlDom_l) \times L_c^2(\nlDom_{nl}\cup \nlBound_{nl})}^2 \leq C \LtNOp(t, u^t, u^t) = C | \int_{\nlDom}\weakSol^t f_{\shape}^t \xt d\xb| \leq C ||f_{\shape}^t\xt||_{L^2(\nlDom)} ||u^t||_{L^2(\nlDom)} \nonumber \\
			&\Rightarrow ||u^t||_{H^1(\nlDom_l) \times L_c^2(\nlDom_{nl} \cup \nlBound_{nl})} \leq C_1,\label{eq:weakSol_bounded}
		\end{align}
		where we used in the last step, that $||u^t||_{L^2(\nlDom)} \leq ||u^t||_{H^1(\nlDom_l) \times L_c^2(\nlDom_{nl}\cup \nlBound_{nl})}$ and  $||f_{\shape}^t\xt||_{L^2(\nlDom)} \rightarrow ||f_{\shape}||_{L^2(\nlDom)}$(Lemma \ref{lemma:l2_convergence}), such that consequently $\{||f_{\shape}^t\xt||_{L^2(\nlDom)}\}_{t \in [0,T]}$ is bounded and \eqref{eq:weakSol_bounded} holds for some $C_1 \in (0,\infty)$.
		Additionally, since also $\{||\data^t||_{L^2(\nlDom)}\}_{t \in [0,T]}$ is bounded from above by a positive constant as a consequence of Lemma \ref{lemma:l2_convergence}, we similarly derive the existence of a constant $C_2 \in (0, \infty)$ with
		\begin{align*}
			&||\advar^t||^2_{H^1(\nlDom_l) \times L_c^2(\nlDom_{nl} \cup \nlBound_{nl})} \leq C \left| \left((\frac{1}{2}(u^t + u^0) - \data^t)\xt, \advar^t \right)_{L^2(\nlDom)} \right| \leq \bar{\xi} C ||\frac{1}{2}(u^t + u^0) - \data^t||_{L^2(\nlDom)} ||\advar^t||_{L^2(\nlDom)}\\
			&\Rightarrow ||\advar^t||_{H^1(\nlDom_l) \times L_c^2(\nlDom_{nl} \cup \nlBound_{nl})} \leq C_2.
		\end{align*}
		\textbf{Assumption (H1):}\\
		In this part of the proof we make use of the following observation:\\
		Given a domain $\widehat{\nlDom} \subset \Rd$ for $d \in \N$. Let a sequence $\{t_k\}_{k \in \N} \in [0,T]^{\N}$ with $t_k \rightarrow 0$, if $k \rightarrow \infty$. Moreover, assume $g^t$,$h^{t_k}, h^0 \in L^2(\widehat{\nlDom})$ for $t \in [0,T]$, $k \in \N$ and $h^{t_k} \rightharpoonup h^0 \in L^2(\widehat{\nlDom})$ for $k \rightarrow \infty$ as well as $g^{t} \rightarrow g^0 \in L^2(\widehat{\nlDom})$, if $t \searrow 0$. Then we get
		\begin{align}
			\begin{split}
				\label{weak_convergence_lemma_2}
				&\left| \int_{\widehat{\nlDom}} h^{t_k}g^t - h^0g^0 ~d\xb \right| \leq \left| \int_{\widehat{\nlDom}} h^{t_k}\left( g^t - g^0 \right) ~d\xb \right| + \left| \int_{\widehat{\nlDom}} \left( h^{t_k} - h^0 \right)g^0 ~d\xb \right| \\
				&\leq ||h^{t_k}||_{L^2(\widehat{\nlDom})} ||g^t - g^0||_{L^2(\widehat{\nlDom})} + \left| \int_{\widehat{\nlDom}} \left( h^{t_k} - h^0 \right) g^0 ~d\xb \right| \rightarrow 0 \text{ for } t \searrow 0 \text{ and } k \rightarrow \infty.
			\end{split}
		\end{align}\\
		As a first step we will show that $u^t \rightharpoonup u^0$ and $\advar^t \rightharpoonup \advar^0$ in $H^1(\nlDom_l) \times L_c^2(\nlDom_{nl} \cup \nlBound_{nl})$ and consequently also in $\LtNSpace(\shape)$.\\
		Due to the boundedness of $u^t$ and $\advar^t$ in $H^1(\nlDom_l) \times L_c^2(\nlDom_{nl} \cup \nlBound_{nl})$ for $t \in [0,T]$ there exists for every sequence $\{ t_n \}_{n \in \N}$ with $t_n \rightarrow 0$ for $n \rightarrow \infty$ two subsequences $\{t_{n_k}\}_{k\in \N}$, $\{t_{n_l}\}_{l\in \N}$ and two functions $q_1, q_2 \in H^1(\nlDom_l) \times L_c^2(\nlDom_{nl} \cup \nlBound_{nl})$ such that $u^{t_{n_l}} \rightharpoonup q_1$ and $\advar^{t_{n_k}} \rightharpoonup q_2$ in $H^1(\nlDom_l) \times L_c^2(\nlDom_{nl} \cup \nlBound_{nl})$.\\
		With Lemma \ref{lemma:l2_convergence} it holds that $\kernelt \rightarrow \kernel$ in $L^2(\nlDom_{nl} \times (\nlDom_{nl} \cup \nlBound_{nl}))$ and thus 
		\begin{align*}
			\psi^t(\xb, \yb) \defas \left( \advar(\xb) - \advar(\yb) \right)\kernelt(\xb,\yb)\xt(\xb)\xt(\yb) \rightarrow \left( \advar(\xb) - \advar(\yb) \right)\kernel(\xb,\yb)= \psi^0(\xb,\yb)
		\end{align*}
		pointwise for a.e. $(\xb,\yb) \in \nlDom_{nl} \times (\nlDom_{nl} \cup \nlBound_{nl})$. Since the kernel $\kernelt$ and $\xt$ are bounded, we derive by applying the dominated convergence theorem that $\psi^t \rightarrow \psi^0$ in $L^2(\nlDom_{nl} \times (\nlDom_{nl} \cup \nlBound_{nl}))$.
		Therefore, by employing \eqref{weak_convergence_lemma_2} we can conclude
		\begin{align*}
			&\int_{\nlDom_{nl}} \int_{\nlDom_{nl} \cup \nlBound_{nl}} \left( u^{t_{n_l}}(\xb) - u^{t_{n_l}}(\yb) \right) \left( \advar(\xb) - \advar(\yb) \right) \kernel^{t_{n_l}}(\xb,\yb) \xi^{t_{n_l}}(\xb) \xi^{t_{n_l}}(\yb) ~d\yb d\xb \\
			&= \int_{\nlDom_{nl}} \int_{\nlDom_{nl} \cup \nlBound_{nl}} \left(u^{t_{n_l}}(\xb) - u^{t_{n_l}}(\yb) \right) \psi^t(\xb,\yb) ~d\yb d\xb \rightarrow \int_{\nlDom_{nl}} \int_{\nlDom_{nl} \cup \nlBound_{nl}} \left( q_1(\xb) - q_1(\yb) \right) \psi^0(\xb,\yb) ~d\yb d\xb \\
			&= \int_{\nlDom_{nl}} \int_{\nlDom_{nl} \cup \nlBound_{nl}} \left( q_1(\xb) - q_1(\yb) \right) \left( \advar(\xb) - \advar(\yb) \right)   \kernel(\xb,\yb) ~d\yb ~d\xb.
		\end{align*}
		Moreover, we can directly follow by using \eqref{weak_convergence_lemma_2} and the continuity of $\transformationMatrix(t)$ in $t$ that
		\begin{align*}
			\int_{\nlDom_l} \left( \transformationMatrix(t) \grad \advar, \grad u^{t_{n_l}}\right) ~d\xb \rightarrow \int_{\nlDom_l} \left(\grad \advar, \grad q_1 \right) ~d\xb.
		\end{align*}
		Since $f_{\shape}^t\xt \rightarrow f_{\shape}$ in $L^2(\nlDom)$ due to Lemma \ref{lemma:l2_convergence} we get
		\begin{align*}
			\LtNOp(0,q_1,\advar) = \lim_{l \rightarrow \infty} \LtNOp(t_{n_l}, u^{t_{n_l}}, \advar) = \lim_{l \rightarrow \infty} \int_{\nlDom} f_{\shape}^{t_{n_l}} \xi^{t_{n_l}} \advar ~d\xb = \int_{\nlDom} f_{\shape} \advar ~d\xb.
		\end{align*}
		Because the solution is unique, we derive $u^0 = q_1$ and $u^t \rightharpoonup u^0$. Analogously, one can show
		\begin{align*}
			\LtNOp(0,\hat{u}, q_2) = \lim_{k \rightarrow \infty} \LtNOp(t_{n_k}, \hat{u}, \advar^{t_{n_k}}) &= - \lim_{k \rightarrow \infty} \int_{\nlDom} \left( \frac{1}{2} \left( u^{t_{n_k}} + u^0 \right) - \bar{u}^{t_{n_k}} \right)\hat{u} \xi^{t_{n_k}} ~d\xb \\ 
			&= - \int_{\nlDom} \left( u^0 - \bar{u} \right) \hat{u} ~d\xb,
		\end{align*}
		which yields $q_2=\advar^0$ and $\advar^t \rightharpoonup \advar^0$ for $t \searrow 0$.\\
		In order to show Assumption (H1) we make use of the following observation: The mean value theorem yields the existence of an $s_t \in (0,t)$ such that
		\begin{align*}
			\frac{G(t,u^0,\advar^t) - G(0,u^0,\advar^t)}{t} = \partial_t G(s_t, u^0, \advar^t).
		\end{align*}
		Therefore, if
		\begin{align}
			\label{ass_h1_convergence}
			\lim_{s,t \searrow 0} \partial_t G(t, u^0, \advar^s) = \partial_t G(0,u^0,\advar^0)
		\end{align}
		holds, Assumption (H1) is fulfilled. Since $\partial_t G(t,u^0\advar^s) = \partial_t \LtNOp(t, u^0, \advar^s) - \partial_t \varForce(t,\advar^s) + \partial_t \objFunc(t,u^0)$ we show the convergence \eqref{ass_h1_convergence} in three steps.\\ 
		First, we want to mention that \mbox{$\kernel \in W^{1,\infty}((\completeDom) \times (\completeDom)) \subset W^{1,1}((\completeDom) \times (\completeDom))$}, \\ \mbox{$f_{\shape} \in H^1(\nlDom) \subset W^{1,1}(\nlDom)$}, such that we can apply Lemma \ref{lemma_frechet_diff_bounded_domain} or Corollary \ref{cor:frechet_diff}, and that the functions $\xt$ and $\transformationMatrix(t)$ are continuously differentiable for $t \in [0,T]$. As a result, the following partial derivatives are all derived by 
		by applying the product rule of Fr\'{e}chet derivatives in $L^1(\nlDom)$ and we just have to show the convergence for $t \searrow 0$.
		Thus, we can conclude by again using \eqref{weak_convergence_lemma_2} that
		\begin{align*}
			\partial_t F(t, \advar^s) &= \int_{\nlDom} (\grad f_{\shape}^t)^T \Vb\advar^s \xt + f_{\shape}^t\advar^s \left. \frac{d}{dr} \right|_{r=t^+} \xi^r ~d\xb \\ 
			&\rightarrow \int_{\nlDom} (\grad f_{\shape})^T\Vb\advar^0 + f_{\shape} \advar^0 \di \Vb ~d\xb = \partial_t \varForce(0,\advar^0), 
		\end{align*}
		since $\grad f_{\shape}^t \xt \rightarrow \grad f_{\shape}$ in $L^2(\nlDom,\Rd)$, $f_{\shape}^t \rightarrow f_{\shape}$ in $L^2(\nlDom)$ by utilizing Lemma \ref{lemma:l2_convergence} and $(\grad f_{\shape}^t)^\top \Vb \xt \rightarrow (\grad f_{\shape})^\top \Vb$ in $L^2(\nlDom)$ due to the boundedness of $\Vb$.
		Then, $\lim_{t \searrow 0} \partial_t \objFunc(t,u^0) = \partial_t \objFunc(0,u^0)$ can be proved in a similar manner.
		In the next step, we investigate the partial derivative of the bilinear form $\LtNOp$ as
		\begin{align*}
			&\partial_t \LtNOp(t, u^0, \advar^s)
			= \int\limits_{\nlDom_l} \left( \transformationMatrix'(t)(\xb) \grad u^0(\xb), \grad \advar^s(\xb) \right) ~d\xb \\ 
			&+ \int\limits_{\nlDom_{nl}}  \int\limits_{\nlDom_{nl} \cup \nlBound_{nl}} \left(\advar^s(\xb) - \advar^s(\yb)\right) \left( u^0(\xb) - u^0(\yb) \right) \kernelt(\xb,\yb) \left. \frac{d}{dr} \right|_{r=t^+} \left( \xi^r(\xb) \xi^r(\yb) \right) ~d\yb d\xb \\
			&+ \int\limits_{\nlDom_{nl}} \int\limits_{\nlDom_{nl} \cup \nlBound_{nl}} \left(\advar^s(\xb) - \advar^s(\yb)\right) \left( u^0(\xb) - u^0(\yb) \right) \left( \grad_{\xb} \kernelt(\xb,\yb) \Vb(\xb) + \grad_{\yb} \kernelt(\xb,\yb)\Vb(\yb) \right)\xt(\xb) \xt(\yb) ~d\yb d\xb,
		\end{align*}
		where we used that $t \rightarrow \transformationMatrix(t)$ is continuously differentiable for every $\xb \in \nlDom_l$ with(see \cite[Lemma 2.14]{Sturm_diss})
		\begin{align*}
			\transformationMatrix'(t)(\xb) =& \tr \left( D\Vb(\xb) \left(D\Ftb(\xb)\right)^{-1} \right)\transformationMatrix(t)(\xb) \\
			&-\left(D\Ftb(\xb)\right)^{-1}D\Vb(\xb)\transformationMatrix(t)(\xb) - \left( \left(D\Ftb(\xb)\right)^{-1}D\Vb(\xb)\transformationMatrix(t)(\xb) \right)^\top,
		\end{align*} 
		Therefore, $\transformationMatrix'(t)\grad u^0 \rightarrow\transformationMatrix'(0)\grad u^0$ in $L^2(\nlDom_l,\Rd)$ due to the dominated convergence theorem, and then \eqref{weak_convergence_lemma_2} yields
		\begin{align*}
			\lim_{s,t \searrow 0} \int_{\nlDom_l} \left( \transformationMatrix'(t)(\xb) \grad u^0(\xb), \grad \advar^s(\xb) \right) ~d\xb = \int_{\nlDom_l} \left( \transformationMatrix'(0)(\xb) \grad u^0(\xb), \grad \advar^0(\xb) \right) ~d\xb.
		\end{align*}
		Moreover, since $\kernel^t(\xb, \yb) \rightarrow \kernel(\xb, \yb)$ in $L^2(\nlDom_{nl} \times \left(\nlDom_{nl} \cup \nlBound_{nl}\right))$, $\grad \kernel^t(\xb, \yb) \rightarrow \grad \kernel(\xb, \yb)$ in \\
		$L^2(\nlDom_{nl} \times \left(\nlDom_{nl} \cup \nlBound_{nl}\right),\Rd \times \Rd)$ and $\xt(\xb)$ continuous in $t$ for almost every $\xb \in \nlDom_{nl}$ we get 
		\begin{align*}
			&\left( u^0(\xb) - u^0(\yb) \right) \kernelt(\xb,\yb) \left. \frac{d}{dr} \right|_{r=t^+} \left( \xi^r(\xb) \xi^r(\yb) \right) \\
			&\rightarrow \left( u^0(\xb) - u^0(\yb) \right) \kernel(\xb,\yb) \left( \di \Vb(\xb) + \di \Vb(\yb) \right) \text{ in } L^2(\nlDom_{nl} \times \left(\nlDom_{nl} \cup \nlBound_{nl} \right) ) \text{ for } t \searrow 0 \text{ and} \\ 
			&\left( u^0(\xb) - u^0(\yb) \right) \left( \grad_{\xb} \kernelt(\xb,\yb) \Vb(\xb) + \grad_{\yb} \kernelt(\xb,\yb)\Vb(\yb) \right)\xt(\xb) \xt(\yb) \\
			&\rightarrow \left( u^0(\xb) - u^0(\yb) \right) \left( \grad_{\xb} \kernel(\xb,\yb) \Vb(\xb) + \grad_{\yb} \kernel(\xb,\yb)\Vb(\yb) \right) \text{ in } L^2(\nlDom_{nl} \times \left(\nlDom_{nl} \cup \nlBound_{nl} \right) ) \text{ for } t \searrow 0.
		\end{align*}
		Again using \eqref{weak_convergence_lemma_2} yields $ \lim_{s,t \searrow 0} \partial_t \LtNOp(t,u^0,\advar^s) = \LtNOp(0,u^0,\advar^0)$.
	\end{proof}
\end{document}